\RequirePackage{fix-cm}
\documentclass[smallextended]{svjour3}       
\smartqed  
\usepackage{amsmath}
\usepackage{amssymb}
\usepackage{stackrel}
\usepackage{graphicx}
\usepackage{epstopdf}
\usepackage{bbm}
\usepackage{enumerate}
\usepackage{etoolbox}
\usepackage[ruled,lined,linesnumbered,vlined]{algorithm2e}
\usepackage{tikz}
\usepackage{color}
\usepackage{url}

\newcommand{\bw}{{\bf w}}
\newcommand{\calI}{\mathcal{I}}
\newcommand{\bcalH}{\boldsymbol{\mathcal{H}}}

\newcommand{\col}[1]{#1}

\usepackage{defs}
\usepackage{prjohns2MacrosTA}
\usepackage{jemacros}
\newtheorem{assumption}{Assumption}
\newcommand{\ourqed}{\qed}
\newcommand{\bcH}{\boldsymbol{\mathcal{H}}}

\begin{document}

\title{
	\col{Projective Splitting with Forward Steps: Asynchronous and Block-Iterative Operator Splitting}
}


\author{Patrick R. Johnstone         \and
        Jonathan Eckstein
}


\institute{
	Patrick R. Johnstone and Jonathan Eckstein \at 
	Department of Management Sciences and Information Systems\\ 
	Rutgers Business School Newark and New Brunswick, Rutgers University\\
	\email{patrick.r.johnstone@gmail.com, jeckstei@business.rutgers.edu}
}

\date{Received: date / Accepted: date}


\maketitle
\allowdisplaybreaks

\begin{abstract}
This work is concerned with the classical problem of finding a zero of a sum
of maximal monotone operators. For the projective splitting framework recently
proposed by Combettes and Eckstein, we show how to replace the fundamental
subproblem calculation using a backward step with one based on two forward
steps. The resulting algorithms have the same kind of coordination procedure
and can be implemented in the same block-iterative and \col{highly flexible} manner, but may perform backward steps on
some operators and forward steps on others. Prior algorithms in the projective
splitting family have used only backward steps. Forward steps can be used for
any Lipschitz-continuous operators provided the stepsize is bounded by the
inverse of the Lipschitz constant. If the Lipschitz constant is unknown, a
simple backtracking linesearch procedure may be used. For affine operators,
the stepsize can be chosen adaptively without knowledge of the Lipschitz
constant and without any additional forward steps.
We close the paper by
empirically studying the performance of several kinds of splitting
algorithms on a \col{large-scale rare feature selection problem}. 
\end{abstract}

\section{Introduction}
For a collection of real Hilbert spaces $\{\calH_i\}_{i=0}^{n}$, consider the problem of finding $z\in\calH_{0}$
such that 
\begin{eqnarray}\label{ProbMono}\label{probMono}
0\in \sum_{i=1}^{n} G_i^* T_i(G_i z),
\end{eqnarray}
where $G_i:\calH_{0}\to\calH_i$ are linear and bounded operators, $T_i:\calH_i \to 2^{\calH_i}$ are maximal monotone
operators and additionally there exists a subset $\Iforw\subseteq\{1,\ldots,n\}$ such that for all $i\in\Iforw$ the operator $T_i$ is Lipschitz continuous.
An important instance of this problem is
\begin{eqnarray}\label{ProbOpt}
\min_{x\in\mathcal{H}_{0}} \sum_{i=1}^{n}f_i(G_i x),
\end{eqnarray}
where every $f_i:\calH_i\to\mathbb{R}$ is closed, proper and convex, with some subset of
the functions also being differentiable with Lipschitz-continuous gradients.
Under appropriate constraint qualifications, \eqref{ProbMono} and
\eqref{ProbOpt} are equivalent.  Problem~\eqref{ProbOpt} arises in a host of
applications such as machine learning, signal and image processing, inverse
problems, and computer vision; 
see~\cite{boyd2011distributed,combettes2011proximal,combettes2005signal} 
for some examples. 
Operator splitting algorithms
are now a common way to solve structured monotone inclusions such
as~\eqref{ProbMono}. Until recently, there were three underlying classes of
operator splitting algorithms: forward-backward \cite{mercier1979lectures},
Douglas/Peaceman-Rachford \cite{lions1979splitting}, and forward-backward-forward
\cite{tseng2000modified}.  In \cite{davis2015three}, Davis and Yin introduced a new
operator splitting algorithm which does not reduce to any of these methods.
Many algorithms for more complicated monotone inclusions and optimization
problems involving many terms and constraints are in fact applications of one
of these underlying techniques to a reduced monotone 
inclusion in an appropriately defined product space~\cite{chambolle2011first,komodakis2015playing,condat2013primal,briceno2011monotone+,combettes2014variable}. These
four operator splitting techniques are, in turn, a special case of the
\emph{Krasnoselskii-Mann (KM) iteration} for finding a fixed point of a
nonexpansive operator \cite{krasnosel1955two,mann1953mean}.


A different, relatively recently proposed class of operator splitting
algorithms is \emph{projective splitting}: this class has a different
convergence mechanism based on projection onto separating sets and does not in
general reduce to the KM iteration. The root ideas underlying projective
splitting can be found in
\cite{iusem1997variant,solodov1999hybrid,solodov1999new}, which dealt with
monotone inclusions with a single operator. The algorithm
of~\cite{eckstein2008family} significantly built on these ideas to address the
case of two operators and was thus the original projective ``splitting"
method. This algorithm was generalized to more than two operators
in~\cite{eckstein2009general}. The related algorithm
in~\cite{alotaibi2014solving} introduced a technique for handling compositions
of linear and monotone operators, and~\cite{combettes2016async} proposed an
extension to ``block-iterative'' and asynchronous operation ---
block-iterative operation meaning that only a subset of the operators making
up the problem need to be considered at each iteration (this approach may be
called ``incremental" in the optimization literature).  A restricted and
simplified version of this framework appears in~\cite{eckstein2017simplified}.
\col{The potentially
asynchronous and block-iterative nature of projective splitting as well as
its ability to handle composition with linear operators gives it an
unprecedented
level of flexibility compared with prior classes of operator
splitting methods.}
Further, in the projective splitting 
methods of~\cite{combettes2016async,eckstein2017simplified} the
order with which operators can be processed is deterministic, variable, and
highly flexible. It is not necessary that each operator be processed the same
number of times either exactly or approximately; in fact, one operator may be
processed much more often than another. The only constraint is that there is
an upper bound on the number of iterations between the consecutive times that
each operator is processed.

Projective splitting algorithms work by performing separate calculations
on each individual operator to construct a separating hyperplane
between the current iterate and the
problem's \emph{Kuhn-Tucker set} (essentially the set of primal and dual
solutions), and then projecting onto this hyperplane. In prior
projective splitting algorithms,
the only operation performed on the individual operators $T_i$ 
is a proximal step \col{(equivalently referred to as a resolvent or backward step)}, 
which consists of evaluating
the operator resolvents $(I + \rho T_i)^{-1}$ for
some scalar $\rho > 0$.  In this paper, we show how, for the Lipschitz continuous
operators, the same kind of framework can also make use of forward steps on
the individual operators, equivalent to applying $I - \rho T_i$.
Typically, such ``explicit'' steps are computationally much easier
than ``implicit'', proximal steps.  Our procedure requires two forward steps
each time it evaluates an operator, and in this sense is reminiscent of
Tseng's forward-backward-forward method~\cite{tseng2000modified} and 
Korpelevich's extragradient method~\cite{korpelevich1976extragradient}. 
Indeed, for the special case of only one operator, projective splitting with
the new procedure reduces to the variant of the extragradient method in
\cite{iusem1997variant} (see \cite[Section 4]{johnstone2018convergence} 
for the derivation).
\col{In our forward-step procedure,}
each stepsize must be bounded by the inverse of the Lipschitz constant of $T_i$.
However, a simple backtracking procedure can eliminate the need to estimate the
Lipschitz constant, and other options are available for 
selecting the stepsize when $T_i$ is affine.


\subsection{Intuition and contributions: basic idea}
\label{basicidea}
We first provide some intuition into our fundamental idea of incorporating
forward steps into projective splitting.  For simplicity, consider (\ref{probMono}) without the linear operators $G_i$, 
that is, we want to find $z$ such that
$0\in \sum_{i=1}^n T_i z$, where $T_1,\ldots,T_n: \calH_0 \rightarrow 2^{\calH_0}$ are
maximal monotone operators on a single real Hilbert space $\calH_0$. 
We formulate the Kuhn-Tucker solution set of this problem as
\begin{align} 
\calS = \set{(z,w_1,\ldots,w_{n-1})}{
w_i\in T_i z,\;i=1,\ldots,n-1,
                                 -\sum_{i=1}^{n-1} w_i \in T_n z}.
                                 \label{ktsetnoGi}
\end{align}
It is clear that $z^*$ solves $0\in \sum_{i=1}^n T_i z^*$ if and only if there
exist $w_1^*,\ldots,w_{n-1}^*$ such that $(z^*,w_1^*,\ldots,w_{n-1}^*)\in \calS$.
A separator-projection 
algorithm for finding a point in $\calS$ finds, at each iteration $k$,
a closed and convex set $H_k$ which separates $\calS$ from the current point,
meaning $\calS$ is entirely in the set and the current point is not. One can then
move closer to the solution set by projecting the current point onto the
set $H_k$.

If we define $\calS$ as in~\eqref{ktsetnoGi}, then the separator formulation 
presented in~\cite{combettes2016async} constructs 
the set $H_k$ through the function 
\begin{align} \label{simplesep}
\varphi_k(z,w_1,\ldots,w_{n-1}) &= \sum_{i=1}^{n-1}\inner{z-x^k_i}{y^k_i-w_i} 
                                    +\Inner{z-x_i^n}{y_i^n + \sum_{i=1}^{n-1} w_i}\\
      &= \Inner{z}{\sum_{i=1}^n y_i^k} + \sum_{i=1}^{n-1}\inner{x_i^k - x_n^k}{w_i}
          - \sum_{i=1}^n \inner{x_i^k}{y_i^k}, \label{simplesepaffine}
\end{align} 
for some $x_i^k,y_i^k \in \calH_0$ such that $y^k_i\in T_i x^k_i$, $i \in
1,\ldots,n$. From its expression in~\eqref{simplesepaffine} it is clear that
$\varphi_k$ is an affine function on $\calH^n_0$.  Furthermore,
it may easily be verified that for any $p = (z,w_1,\ldots,w_{n-1}) \in \calS$, one
has $\varphi_k(p)\leq 0$, so that the separator set $H_k$ may be taken to be
the halfspace $\set{p}{\varphi_k(p) \leq 0}$. The key idea of projective
splitting is, given a current iterate $p^k = (z^k,w_1^k,\ldots,w_{n-1}^k) \in
\calH_0^n$, to pick $(x^k_i,y^k_i)$ so that $\varphi_k(p^k)$ is positive if $p^k
\not\in \calS$. Then, since the solution set is entirely on the other side of
the hyperplane $\set{p}{\varphi_k(p)=0}$, projecting the current point onto
this hyperplane makes progress toward the solution.  If it can be shown that this
progress is sufficiently large, then it is possible to prove (weak) convergence.

Let the iterates of such an algorithm be $p^k = (z^k,w_i^k,\ldots,w_{n-1}^k) \in
\calH_0^n$.  
To simplify the subsequent analysis, define $w_n^k \triangleq
-\sum_{i=1}^{n-1} w_i^k$ at each iteration $k$, whence it is immediate from~\eqref{simplesep} 
that $\varphi_k(p^k) = \varphi_k(z^k,w_1^k,\ldots,w_{n-1}^k) = \sum_{i=1}^n
\inner{z^k-x^k_i}{y^k_i-w^k_i}$.
To construct a function $\varphi_k$ of the form~\eqref{simplesep} such
that $\varphi_k(p^k) = \varphi_k(z^k,w_1^k,\ldots,w_n^k) > 0$ whenever $p^k \not\in
\calS$, it is sufficient to be able to perform the following calculation on
each individual operator $T_i$: for $(z^k,w_i^k)\in \calH_0^2$, find
$x_i^k,y_i^k\in\calH_0$ such that $y_i^k\in T_i x^k_i$ and $\inner{z^k - x_i^k}{y_i^k -
w_i^k} \geq 0$, with
$\inner{z^k - x_i^k}{y_i^k - w_i^k} > 0$ if $w_i^k
\not\in T_i z^k$.  In earlier work on projective
splitting~\cite{eckstein2008family,eckstein2009general,combettes2016async,alotaibi2014solving}, the calculation 
of such a $(x_i^k,y_i^k)$ 
is accomplished by a proximal (implicit) step on the operator $T_i$: given a
scalar $\rho > 0$, we find the unique pair $(x_i^k,y_i^k)\in\calH_0^2$ such that
$y_i^k \in T_i x^k_i$ and
\begin{equation} \label{backwardidentity}
x_i^k+\rho y_i^k = z^k+\rho w_i^k 
\quad \Rightarrow \quad 
z^k-x_i^k = \rho(y_i^k-w_i^k).
\end{equation}
We immediately conclude that 
\begin{equation} \label{backwardresult}
\inner{z^k - x_i^k}{y_i^k - w_i^k} = (1/\rho)\smallnorm{z^k - x_i^k}^2 \geq 0,
\end{equation}
and furthermore that $\inner{z^k - x_i^k}{y_i^k - w_i^k} > 0$ unless $x_i^k =
z^k$, which would in turn imply that $y_i^k = w_i^k$ and $w_i^k \in T_i z^k$.  If
we perform such a calculation for each $i= 1,\ldots,n$, we have constructed a
separator of the form~\eqref{simplesep} which, in view of 
$\varphi_k(p^k) = \sum_{i=1}^n \inner{z^k-x^k_i}{y^k_i-w_i^k}$,
has $\varphi_k(p^k) > 0$ if $p^k
\not\in \calS$. This basic calculation on $T_i$ is depicted in
Figure~\ref{fig:visualize}(a) for $\calH_0 = \real^1$: 
because $z^k-x_i^k =
\rho(y_i^k-w_i^k)$, the line segment between $(z^k,w_i^k)$ and $(x_i^k,y_i^k)$
must have slope $-1/\rho$, meaning that $\inner{z^k - x_i^k}{w_i^k - y_i^k}
\leq 0$ and thus that $\inner{z^k - x_i^k}{y_i^k - w_i^k} \geq 0$.  It also
bears mentioning that the relation~\eqref{backwardresult} plays (in
generalized form) a key role in the convergence proof.

Consider now the case that $T_i$ is Lipschitz continuous with modulus $L_i
\geq 0$ (and hence single valued) and defined throughout $\calH_0$.  We now
introduce a technique to accomplish something similar to the preceding
calculation through two forward steps instead of a single backward step.  We
begin by evaluating $T_i z^k$ and using this value in place of $y_i^k$ in the
right-hand equation in~\eqref{backwardidentity}, yielding
\begin{equation} \label{forwardidentity}
z^k-x_i^k = \rho\big(T_i z^k-w_i^k\big)
\quad \Rightarrow \quad
x_i^k = z^k - \rho\big(T_i z^k-w_i^k\big),
\end{equation}
and we use this value for $x_i^k$.  This calculation is depicted by the lower
left point in Figure~\ref{fig:visualize}(b).  We then calculate $y_i^k =
T_i x^k_i$, resulting in a pair $(x_i^k,y_i^k)$ on the graph of the operator;
see the upper left point in Figure~\ref{fig:visualize}(b). For this choice of $(x_i^k,y_i^k)$, we next observe that
\begin{align}
\inner{z^k - x_i^k}{y_i^k - w_i^k}
&=
\Inner{z^k - x_i^k}{T_i z^k - w_i^k} - \Inner{z^k - x_i^k}{T_i z^k - y_i^k} \nonumber \\
&=
\Inner{z^k - x_i^k}{\tfrac{1}{\rho}(z^k - x_i^k)} 
- \Inner{z^k - x_i^k}{T_i z^k - T_i x^k_i} \label{fstep1}\\
&\geq
\tfrac{1}{\rho} \norm{z^k - x_i^k}^2 - L_i \norm{z^k - x_i^k}^2 \label{fstep2} \\
&=
\left(\frac{1}{\rho} - L_i\right) \norm{z^k - x_i^k}^2. \label{fstep3}
\end{align}
Here, \eqref{fstep1} follows because $T_i z^k - w_i^k = (1/\rho)(z^k - x_i^k)$
from~\eqref{forwardidentity} and because we let $y_i^k = T_i x^k_i$. The
inequality~\eqref{fstep2} then follows from the Cauchy-Schwarz inequality and the hypothesized
Lipschitz continuity of $T_i$.  If we require that $\rho < 1/L_i$, then we
have $1/\rho > L_i$ and~\eqref{fstep3} therefore establishes that $\inner{z^k -
	x_i^k}{y_i^k - w_i^k} \geq 0$, with $\inner{z^k - x_i^k}{y_i^k - w_i^k} > 0$
unless $x_i^k = z^k$, which would imply that $w_i^k = T_i z^k$.  We thus
obtain a conclusion very similar to~\eqref{backwardresult} and the results
immediately following from it, but using the constant $1/\rho - L_i > 0$ in place of
the positive constant $1/\rho$.

\begin{figure}
\begin{center}
\begin{tikzpicture}[scale=0.68]
      \node at (-2,4) {(a)} ;
      \draw[->] (-3,-0.5) -- (4.2,-0.5) ;  
      \draw[->] (-1,-1) -- (-1,4.2) ;  
      \draw[domain=-3:4.2,smooth,variable=\x,blue,thick] 
          plot ({\x},{0.5*\x + 1 + 0.015*\x*\x*\x});
      \node at (3.5,4) {{\color{blue}$T_i$}};
      \draw [->,thick] (3,0.5) -- (2.05,2.02); 
      \node at (2.0,0.6) {{\tiny $-1/\rho$}};
      \draw [dotted] (2.2,1.74) -- (2.2,0.8) ;
      \draw [dotted] (2.2,0.8) -- (2.84,0.8) ;
      \draw [fill=black] (3,0.5) circle [radius=0.05] node[right]{$(z^k,w_i^k)$};
      \draw [fill=black] (2,2.12) circle [radius=0.05] node[right]{~~$(x_i^k,y_i^k)$};
      \node at (7,4) {(b)} ;
      \draw[->] (6,-0.5) -- (13.2,-0.5) ;  
      \draw[->] (8,-1) -- (8,4.2) ;  
      \draw[domain=6:13.2,smooth,variable=\x,blue,thick] 
          plot ({\x},{0.5*(\x - 9) + 1 + 0.015*(\x - 9)*(\x - 9)*(\x - 9)});
      \node at (12.5,4) {{\color{blue}$T_i$}};
      \draw [fill=black] (12,0.5) circle [radius=0.05] node[right]{$(z^k,w_i^k)$};
      \draw [fill=black] (12,2.90) circle [radius=0.05] 
                              node[left]{$\big(z^k,T_i z^k\big)$~~};
      \draw [->,thick] (12,0.5) -- (12,2.85) ;
      \draw [fill=black] (10,0.5) circle [radius=0.05] node[below]{$(x_i^k,w_i^k)$};
      \draw [->] (12,2.90) -- (10.05,0.55) ;
      \draw [dotted] (11,1.7) -- (11.4,1.7) -- (11.4,2.15);
      \node at (11.5,1.5) {{\tiny $1/\rho$}} ;
      \draw [fill=black] (10.0,1.51) circle [radius=0.05] node[left]{$(x_i^k,y_i^k)$~~~};
      \draw [->,thick] (10,0.5) -- (10,1.5) ;
      \draw [->,dashed] (12,0.5) -- (10.1,1.45) ;
\end{tikzpicture}
\end{center}
\caption{Backward and forward operator calculations in $\calH_0 = \real^1$.
The goal is to find a point $(x_i^k,y_i^k)$ on the graph of the operator such
that line segment connecting $(z^k,w_i^k)$ and $(x_i^k,y_i^k)$ has negative
slope. Part (a) depicts a standard backward-step-based construction, while (b)
depicts our new construction based on two forward steps. }
\label{fig:visualize}
\end{figure}

For $\calH_0 = \real^1$, this process is depicted in
Figure~\ref{fig:visualize}(b).  By construction, the line segment between
$\big(z^k,T_i z^k\big)$ and $(x_i^k,w_i^k)$ has slope $1/\rho$, which is
``steeper'' than the graph of the operator, which can have slope at most $L_i$
by Lipschitz continuity.  This guarantees that the line segment between
$(z^k,w_i^k)$ and $(x_i^k,y_i^k)$ must have negative slope, which in $\real^1$
is equivalent to the claimed inner product property.

Using a backtracking line search, we will also be able to handle the
situation in which the value of $L_i$ is unknown.  If we choose any positive
constant $\Delta > 0$, then by elementary algebra the inequalities $(1/\rho) - L_i \geq
\Delta$ and $\rho \leq 1/(L_i + \Delta)$ are equivalent.  Therefore, if we select some
positive $\rho \leq 1/(L_i + \Delta)$, we have from~\eqref{fstep3} that
\begin{eqnarray}\label{backtrack}
\inner{z^k - x_i^k}{y_i^k-w_i^k} \geq \Delta{\|z^k-x_i^k\|^2},
\end{eqnarray}
which implies the key properties we need for the convergence proofs.
Therefore we may start with any $\rho = \rho^0 > 0$, and repeatedly halve
$\rho$ until~\eqref{backtrack} holds;  
in Section~\ref{secLineSearch} below, we bound the number of halving steps required.
In general, each trial value of $\rho$ requires one application of the
 Lipschitz continuous operator $T_i$.  However, for the case of affine
 operators $T_i$, we will show that it is possible to compute a 
 stepsize such that \eqref{backtrack} holds
with a total of only two applications of the operator. 
%
%
 By contrast, most
 backtracking procedures in optimization algorithms require evaluating the
 objective function at each new candidate point, which in turn usually requires an
 additional matrix multiply operation in the quadratic case~\cite{beck2009fast}. 

\subsection{Summary of Contributions}

The main thrust of the remainder of this paper is to incorporate the second,
forward-step construction of $(x_i^k,y_i^k)$ above into an algorithm
resembling those of~\cite{combettes2016async,eckstein2017simplified}, 
allowing some operators to
use backward steps, and others to use forward steps.  Thus, projective
splitting may become useful in a broad range of applications in which computing 
forward steps is preferable to computing or approximating proximal
steps. The resulting algorithm inherits the block-iterative
features and \col{flexible}
capabilities of~\cite{combettes2016async,eckstein2017simplified}.

We will work with a slight restriction of problem~\eqref{ProbMono}, namely
\begin{eqnarray}\label{probCompMonoMany}
0\in \sum_{i=1}^{n-1} G_i^* T_i(G_i z) +T_n(z).
\end{eqnarray}
In terms of problem~\eqref{ProbMono}, we are simply requiring that $G_n$ be
the identity operator and thus that $\calH_n = \calH_0$. This is not much of a
restriction in practice, since one could redefine the last operator as 
$T_{n}\leftarrow G^*_{n} \circ T_{n} \circ G_{n}$, or one could
simply append a new operator $T_n$ with $T_n(z) = \{0\}$ everywhere. 

The
principle reason for adopting a formulation involving the linear operators
$G_i$ is that in many applications of~\eqref{probCompMonoMany} it may be
relatively easy to compute the proximal step of $T_i$ but difficult to compute
the proximal step of $G_i^*\circ T_i\circ G_i$. Our framework will include
algorithms for~\eqref{probCompMonoMany} that may compute the proximal steps on
$T_i$, forward steps when $T_i$ is Lipschitz continuous, and
applications (``matrix multiplies") of $G_i$ and $G_i^*$.
An interesting feature of the forward steps in our method is that while the
allowable stepsizes depend on the Lipschitz constants of the $T_i$ for
$i\in\Iforw$, they do not depend on the linear operator norms $\|G_i\|$, in
contrast with primal-dual 
methods~\cite{chambolle2011first,condat2013primal,vu2013splitting}.
Furthermore, as already mentioned, 
the stepsizes used for each operator can be chosen
independently and may vary by iteration.



We \col{also present a previously unpublished ``greedy''} 
heuristic for selecting operators in block-iterative
splitting, based on a simple proxy. Augmenting this heuristic with a straightforward
safeguard allows one to retain all of the convergence properties of the main
algorithm.  The heuristic is not specifically tied to the use of forward steps
and also applies to the earlier algorithms
in~\cite{combettes2016async,eckstein2017simplified}. The numerical experiments
in Section~\ref{secNumerical} below attest to its usefulness.

\col{The main 
	contribution of this work is the new two-forward-step procedure. The main
	 proposed algorithm is a block-iterative splitting method that performs
	 well in our numerical experiments when combined with the greedy block
	 selection strategy. However, the analysis also allows for the kind of
	 asynchronous operation developed in
	 \cite{combettes2016async,eckstein2017simplified}. Empirically
	 investigating such asynchronous implementations is beyond the scope of
	 this work. Since allowing for asynchrony introduces little additional
	 complexity into the convergence analysis, we have included it in the
	 theoretical results. }

\col{After submitting this paper, we became aware of the
preprint~\cite{TV15}, which develops a similar two-forward-step
procedure for projective splitting in a somewhat different setting than
\eqref{probCompMonoMany}. The scheme is equivalent to ours when $G_i=I$, but
does not incorporate the backtracking linesearch or its simplification for
affine operators. Their analysis also does not allow for asynchronous or
block-iterative implementations. }


\section{Mathematical Preliminaries}
\subsection{Notation}
Summations of the form $\sum_{i=1}^{n-1}a_i$ for some collection $\{a_i\}$ 
will appear throughout this paper. 
To deal with the case $n=1$, we use the standard convention that
$
\sum_{i=1}^{0} a_i = 0.
$
To simplify the presentation, we use the following notation
throughout the rest of the paper, where $I$ denotes the identity map on
$\mathcal{H}_n$:
\begin{align}
G_{n} &= I
&
\label{defwn}
(\forall\,k\in\mathbb{N}) \;\;\; w_{n}^k &\triangleq - \sum_{i=1}^{n-1} G_i^*w_i^k.
\end{align}
Note that when $n=1$, $w_1^k = 0$. 
We will use a boldface $\bw = (w_1,\ldots,w_{n-1})$ for elements 
of $\calH_1\times\ldots\times\calH_{n-1}$.

Throughout, we will simply write $\|\cdot\|_i = \|\cdot\|$ as the
norm for $\calH_i$ and let the subscript be inferred from the argument. In the
same way, we will write $\langle\cdot,\cdot\rangle_i$ as $\langle\cdot
,\cdot\rangle$ for the inner product of $\calH_i$. For the collective
primal-dual space defined in Section~\ref{secMainAss}, we will use a special norm
and inner product with its own subscript.



For any maximal monotone operator $A$ we will use the notation
$
\prox_{\rho A} = (I+\rho A)^{-1},
$
for any scalar $\rho > 0$, to denote the \emph{proximal operator}, 
also known as the backward or implicit step with respect to $A$.  This means that
\begin{eqnarray}\label{defprox2}
x = \prox_{\rho A}(a) \quad\implies\quad \exists y\in Ax:x+\rho y = a.
\end{eqnarray}
The $x$ and $y$ satisfying this relation are unique. 
Furthermore, $\prox_{\rho A}$ is defined everywhere and
$\text{range}(\prox_A) = \text{dom}(A)$ \cite[Prop. 23.2]{bauschke2011convex}. 



We use the standard ``$\rightharpoonup$" notation to denote weak convergence, which is
of course equivalent to ordinary convergence in finite-dimensional
settings.

The following basic result will be used several times in our proofs:
\begin{lemma}\label{lemBasic}
For any vectors $v_1,\ldots,v_n$, 
$
\left\|\sum_{i=1}^n v_i\right\|^2\leq n\sum_{i=1}^n\left\|v_i\right\|^2.
$
\end{lemma}
\begin{proof}
$
\left\|\sum_{i=1}^n v_i\right\|^2
=
n^2\left\|\frac{1}{n}\sum_{i=1}^n v_i\right\|^2
\leq
n^2\cdot \frac{1}{n}\sum_{i=1}^n\|v_i\|^2,
$
where the inequality follows from the convexity of the function $\|\cdot\|^2$. 
\end{proof}

\subsection{Main Assumptions Regarding Problem (\ref{probCompMonoMany})}\label{secMainAss}
Let $\boldsymbol{\mathcal{H}} = \mathcal{H}_0\times
\mathcal{H}_1\times\cdots\times\mathcal{H}_{n-1}$ and $\calH_n = \calH_0$.
Define the \emph{extended solution set} or \emph{Kuhn-Tucker set}
of~\eqref{probCompMonoMany} to be
\begin{align} 
\nonumber 
\calS
=  \Big\{ (z,w_1,\ldots,w_{n-1}) \in \boldsymbol{\mathcal{H}} \;\; \Big| \;\;
w_i\in T_i(G_i z),\,\, i=1,\ldots,n-1, \;\;
\\
-\sum_{i=1}^{n-1} G_i^* w_i  \in T_n(z) \Big\}.
\label{defCompExtSol}
\end{align}
Clearly $z\in\mathcal{H}_{0}$ solves~\eqref{probCompMonoMany} if and only if
there exists $\bw\in\calH_1 \times \cdots \times \calH_{n-1}$ such that
$(z,\bw)\in\calS$. Our main assumptions regarding~\eqref{probCompMonoMany} are
as follows:
\begin{assumption}
	\label{AssMonoProb}\label{assMono}	
	Problem~(\ref{probCompMonoMany}) conforms to the following:
	\begin{enumerate}
	\item $\mathcal{H}_0 = \mathcal{H}_n$ and
	$\mathcal{H}_1,\ldots,\mathcal{H}_{n-1}$ are real Hilbert spaces.
	\item For $i=1,\ldots,n$, the operators
	$T_i:\mathcal{H}_{i}\to2^{\mathcal{H}_{i}}$ are monotone. 
	\item For all $i$ in some subset $\Iforw \subseteq \{1,\ldots,n\}$,
    the operator $T_i$ is
	$L_i$-Lipschitz continuous (and thus single-valued)
	and $\text{dom}(T_i) = \mathcal{H}_i$.  
	\item For $i \in
	\Iback \triangleq \{1,\ldots,n\} \backslash \Iforw$, the operator 
	$T_i$ is maximal and that the map $\prox_{\rho
	T_i}:\mathcal{H}_i\to\mathcal{H}_i$ can be computed 
 	to within the error tolerance specified below in Assumption \ref{assErr} 
(however, these operators are not precluded 
	from also being Lipschitz continuous).
	\item Each $G_i:\calH_{0}\to\calH_i$ for $i=1,\ldots,n-1$ is
	linear and bounded. 
	\item The solution set $\calS$ defined in
	(\ref{defCompExtSol}) is nonempty.
	\end{enumerate}
\end{assumption}

\begin{lemma}
\label{lemClosed}
Suppose Assumption \ref{AssMonoProb} holds. The set $\calS$ defined in \eqref{defCompExtSol} is closed and convex. 
\end{lemma} 
\begin{proof}
We first remark that for $i\in\Iforw$ the operators $T_i$ are maximal by
\cite[Proposition~20.27]{bauschke2011convex}, so $T_1,\ldots,T_n$ are all maximal
monotone. The claimed result is then a special case of~\cite[Proposition
2.8(i)]{briceno2011monotone+} with the following change of \col{notation, where ``MM'' stands for ``maximal monotone" and ``BL'' stands for ``bounded linear''}:
	\begin{eqnarray*}
	\text{\textbf{Notation here}} & & 
	\text{\textbf{Notation in \cite{briceno2011monotone+}}} \\
	T_n 
	&\longrightarrow &
	A\text{ (MM operator)}
	\\
	(x_1,\ldots ,x_{n-1}) \mapsto T_1 x_1 \times \cdots \times T_{n-1} x_{n-1}
	&\longrightarrow &
	B\text{ (MM operator)}
	\\
	z \mapsto (G_1 z,\ldots,G_{n-1} z)
	&\longrightarrow&
	L\text{ (BL operator)}.
\end{eqnarray*}

~

\vspace{-6.5ex}
\ourqed
\end{proof}
\subsection{A Generic Linear Separator-Projection Method}

Suppose that $\boldsymbol{\mathcal{H}}$ is a real Hilbert space with inner
product $\langle\cdot,\cdot\rangle_{\bcalH}$ and norm $\|\cdot\|_{\bcalH}$. A
generic  linear separator-projection method for finding a point in some closed
and convex set $\calS\subseteq\boldsymbol{\mathcal{H}}$ is given in Algorithm
\ref{AlgGenericProject}.

\begin{algorithm}[h!]\label{AlgGenericProject}
	\DontPrintSemicolon
	\SetKwInOut{Input}{Input}
\Input{$p^1$, $0<\underline{\beta}\leq\overline{\beta}<2$}
	\caption{Generic linear separator-projection method for finding a point in a closed and convex set $\calS\subseteq\boldsymbol{\mathcal{H}}$.}
	\label{AlgAbstractProject}
	\For{$k=1,2,\ldots,$}{
	Find an affine function $\varphi_k$ such that $\nabla\varphi_k\neq 0$ and
	    $\varphi_k(p)\leq 0$ for all $p\in \calS$. \;
	Choose $\beta_k\in[\underline{\beta},\overline{\beta}]$\label{eqProjectUpdatem1}\;
	$p^{k+1} = p^k - 
	\frac{\beta_k\max\{0,\varphi_k(p^k)\}}
	{\|\nabla\varphi_k\|_{\bcalH}^2}
	\nabla\varphi_k$
	\label{eqProjectUpdate}\;
 }
\end{algorithm}

The update on line~\ref{eqProjectUpdate} is the $\beta_k$-relaxed projection of $p^k$
onto the halfspace $\{p:\varphi_k(p) \leq 0\}$ using the norm
$\|\cdot\|_{\bcalH}$. In other words, if $\hat{p}^k$ is the projection onto
this halfspace, then the update is $p^{k+1} = (1-\beta_k)p^k +
\beta_k\hat{p}^k$. Note that we define the gradient $\nabla\varphi_k$ with
respect to the inner product $\langle\cdot,\cdot\rangle_{\bcalH}$, meaning we
can write
$$
(\forall p,\tilde{p}\in\bcalH):\quad\varphi_k(p) = \langle\nabla\varphi_k,p - \tilde{p}\rangle_{\bcalH} + \varphi_k(\tilde{p}).
$$

We will use the following well-known properties
of algorithms fitting the template of Algorithm~\ref{AlgGenericProject}; see for example~\cite{combettes2000fejerXX,eckstein2008family}:
\begin{lemma} 
Suppose $\calS$ is closed and convex. Then for Algorithm \ref{AlgAbstractProject},
	\begin{enumerate}
		\item The sequence $\{p^k\}$ is bounded.
		\item $\|p^k - p^{k+1}\|_{\bcalH}\to 0$\label{pointDiff20};
		\item \label{pointFejer}If all weak limit points of $\{p^k\}$ are in $\calS$, then
		$p^k$ converges weakly to some point in $\calS$.
	\end{enumerate}\label{propFejer}\label{lemFejer}
\end{lemma}
Note that we have not specified how to choose the affine function $\varphi_k$.
For our specific application of the separator-projection
framework, we will do so in Section~\ref{secHplane}.

\subsection{\col{Our Hyperplane}}\label{secTecLems}\label{secHplane}
In this section, we define the affine function our algorithm uses to construct
a separating hyperplane. 
Let $p =(z,\bw) =  (z,w_1,\ldots, w_{n-1})$ be a
generic point in $\bcH$, the collective primal-dual space. For $\bcH$, we
adopt the following norm and inner product for some $\gamma>0$:
\begin{align} \label{gammanorm}
\norm{(z,\bw)}_\gamma^2 &= \gamma\|z\|^2 + \sum_{i=1}^{n-1}\|w_i\|^2 
\\\nonumber 
\Inner{(z^1,\bw^1)}{(z^2,\bw^2)}_\gamma &= 
\gamma\langle z^1,z^2\rangle + \sum_{i=1}^{n-1}\langle
w^1_i,w^2_i\rangle.
\end{align}
Define the following function
generalizing~\eqref{simplesep} at each iteration $k\geq 1$:
\begin{eqnarray}\label{hyper}\label{hplane}\label{defHyper}
\varphi_k(p)
&=&
\sum_{i=1}^{n-1}
\left
\langle G_i z - x_i^k,y_i^k - w_i
\right
\rangle
+
\left\langle
z-x_{n}^k,
y_{n}^k
+
\sum_{i=1}^{n-1} G_i^* w_i
\right\rangle,
\label{defCompHplane}
\end{eqnarray}
where the $(x_i^k,y_i^k)$ are chosen so that $y_i^k\in T_i x_i^k$ for
$i=1,\ldots,n$ (recall that each inner product is for the corresponding
Hilbert space $\calH_i$).
This function is a special case of the separator function used
in~\cite{combettes2016async}.  
The following lemma proves some basic
properties of $\varphi_k$; similar results are
in~\cite{alotaibi2014solving,combettes2016async,eckstein2017simplified} in the
case $\gamma=1$.

\begin{lemma}\label{LemGradAffine} Let $\varphi_k$ be defined as in (\ref{hplane}).  Then:
	\begin{enumerate}
		\item $\varphi_k$ is affine on $\bcH$. 
		\item \label{item:gradForm} With respect to inner product
		$\langle\cdot,\cdot\rangle_\gamma$ on $\bcH$, the gradient
		of $\varphi_k$ is
		$$
		\nabla\varphi_k =
		\left(\frac{1}{\gamma}\left(\sum_{i=1}^{n-1} G_i^* y_i^k+ y_n^k\right) \!\!,\;
		x_1^k - G_1 x_{n}^k,\ldots,x_{n-1}^k - G_{n-1} x_{n}^k\right).
		$$
		\item 
		Suppose Assumption \ref{AssMonoProb} holds and 
		that $y_i^k\in T_i x_i^k$ for $i=1,\ldots,n$. Then
		$\varphi_k(p)\leq 0$ for all $p\in \calS$ defined in (\ref{defCompExtSol}).
		\item \label{lem0grad}
		If Assumption~\ref{AssMonoProb} holds, $y_i^k\in T_i x_i^k$ for $i=1,\ldots,n$, and $\nabla\varphi_k =  0$, then 
		$
		(x_n^k,y_1^k,\ldots,y_{n-1}^k)\in\calS.
		$
	\end{enumerate}
\end{lemma}

\begin{proof}
To see that $\varphi_k$ is affine, rewrite~\eqref{defHyper} as
\begin{align*}
\varphi_k(z,\bw)
&=
\sum_{i=1}^{n-1}
\langle
G_i z,
y_i^k - w_i
\rangle 
-
\sum_{i=1}^{n-1}
\langle
x_i^k,
y_i^k - w_i
\rangle 
+
\left\langle
z,
y_n^k + \sum_{i=1}^{n-1} G_i^* w_i
\right\rangle 
\\
&
\qquad 
-\left\langle
x_n^k,
y_n^k + \sum_{i=1}^{n-1} G_i^* w_i
\right\rangle 
\end{align*}
\pagebreak[3]   
\begin{align}
\qquad\quad\quad &=
\sum_{i=1}^{n-1}
\langle
z,
G_i^*(y_i^k - w_i)
\rangle 
+
\sum_{i=1}^{n-1}\langle 
w_i
,
x_i^k
\rangle
-
\sum_{i=1}^{n}
\langle
x_i^k,
y_i^k
\rangle 
\nonumber\\
&
\qquad
+
\left\langle
z,
y_n^k + \sum_{i=1}^{n-1} G_i^* w_i
\right\rangle 
-
\sum_{i=1}^{n-1}
\left\langle
 w_i
 ,
  G_i x_n^k
\right\rangle
\nonumber\\
\label{eqAffineExpress}
&=
\left\langle
z
,
\sum_{i=1}^{n-1}
G_i^*
y_i^k
+
y_{n}^k
\right\rangle 
+
\sum_{i=1}^{n-1}
\langle
w_i,
x_i^k - G_i x_{n}^k
\rangle
-
\sum_{i=1}^{n}\langle x_i^k,y_i^k\rangle.
\end{align}	
It is now clear that $\varphi_k$ is an affine function of $p = (z,\bw)$. 
Next, fix an arbitrary $\tilde{p}\in\bcalH$.  Using that $\varphi_k$ is affine, we
may write
\begin{eqnarray*}
\varphi_k(p) 
&=& \inner{p - \tilde{p}}{\nabla \varphi_k}_\gamma + \varphi_k(\tilde{p}) =
\langle p,\nabla\varphi_k\rangle_{\gamma}
+\varphi_k(\tilde{p})-\langle\tilde{p},\nabla\varphi_k\rangle_\gamma
\\
&=&
\gamma \langle z,\nabla_z\varphi_k\rangle 
+
\sum_{i=1}^{n-1}
\langle w_i,\nabla_{w_i}\varphi_k\rangle 
+\varphi_k(\tilde{p})-\langle\tilde{p},\nabla\varphi_k\rangle_\gamma
\end{eqnarray*}
Equating terms between this expression and~\eqref{eqAffineExpress} yields the
claimed expression for the gradient.

Next, suppose Assumption~\ref{AssMonoProb} holds and $y_i^k\in T_i x_i^k$
for $i=1,\ldots,n$.  To prove the third claim, we need to consider
$(z,\bw)\in\calS$ and establish that $\varphi_i(z,\bw) \leq 0$.  We do so by
showing that all $n$ terms in~\eqref{defHyper} are nonpositive: first, for
each $i=1,\ldots,n-1$, we have $\inner{G_i z - x_i^k}{y_i^k - w_i} \leq 0$
since $T_i$ is monotone, $w_i\in T_i(G_i z)$, and $y_i^k\in T_i x_i^k$.  The
nonpositivity of the final term is established similarly by noting that
$y_{n}^k \in T_{n} x_{n}^k$, $-\sum_{i=1}^{n-1} G_i^* w_i \in T_n z$, and that
$T_{n}$ is monotone.

Finally, suppose $\nabla\varphi_k = 0$ for some $k\geq 1$. Then
$y_n^k = -\sum_{i=1}^{n-1}G_i^* y_i^k$ and 
$x_i^k - G_i x_n^k = 0$ for all $i=1,\ldots,n-1$. 
The latter implies that  
$y_i^k\in T_i(G_i x_n^k)$ for all $i=1,\ldots,n-1$.
Since we also have $y_n^k\in T_n(x_n^k)$, we obtain
that $(x_n^k,y_1^k,\ldots,y_{n-1}^k)\in\calS$. 
\qed
\end{proof}

\section{Our Algorithm}
\subsection{Algorithm Definition}

\begin{algorithm}[h!]
	\DontPrintSemicolon
\SetKwInOut{Input}{Input}
\SetKwInOut{Parameters}{Parameters}
\Input{$(z^1,{\bf w}^1)\in \boldsymbol{\mathcal{H}}$, $(x_i^0,y_i^0)\in\mathcal{H}_i^2$ for $i=1,\ldots,n$.}
\Parameters{$\{I_k\}_{k\in\nN}$ where $I_k\subseteq\{1,\ldots,n\}$, $\{d(i,k)\}_{k\in\nN}$ for $i=1,\ldots,n$ where $1\leq d(i,k)\leq k$, $0<\underline{\beta}\leq\overline{\beta}<2$, $\gamma>0$.}

\For{$k=1,2,\ldots$}
{   

		\For{$i\in I_k$}
		{
			\tcc{\col{these are the active operators to be processed}}
			\label{lineStartFor}
			\If{$i\in\Iback$}
			{
			    	$a = G_i z^{d(i,k)}+\rho_{i}^{d(i,k)} w_i^{d(i,k)}+e_i^k$\label{lineaupdate}
			    					\tcc*{\col{do a backward step}}
			    $x_i^k = \prox_{\rho_{i}^{d(i,k)} T_i}(a)
			    $\label{LinebackwardUpdate}\;
			    $
			    y_i^k = (\rho_{i}^{d(i,k)})^{-1}
			    \left(
			    a - x_i^k
			    \right)
			    $\label{lineBackwardUpdateY}\;

		    }
		    \Else
			{
				\tcc{\col{do two forward steps}}            	            	
			    \label{LineForwardUpdate}
				$					
				x_i^k = G_i z^{d(i,k)}-\rho_{i}^{d(i,k)}
				( T_i G_i z^{d(i,k)} - w_{i}^{d(i,k)}),
				$\label{ForwardxUpdate} \;
				$
				y_i^k = T_i x_i^k.
				$
				\label{ForwardyUpdate}\;

		    }
        }
	    \For{$i\notin I_k$} 
		{
			\tcc{\col{These are the inactive operators}}
			$(x_i^k,y_i^k)=(x_i^{k-1},y_i^{k-1})$\label{lineLeave}
		}
	
	\tcc{\col{Beginning of projection procedure}} 
    $u_i^k = x_i^k - G_i x_n^k,\quad i=1,\ldots,n-1,$\label{lineCoordStart}\;
	$v^k = \sum_{i=1}^{n-1} G_i^* y_i^k+y_n^k$\label{lineVupdate}\;    
	$\pi_k = \|u^k\|^2+\gamma^{-1}\|v^k\|^2$  \label{linePiUpdate}\;    
	\eIf{$\pi_k>0$}{
		Choose some $\beta_k \in [\underline{\beta},\overline{\beta}]$\;  \label{lineHplane} 
		$\varphi_k(p^k) = 
		\langle z^k, v^k\rangle 
		+
		\sum_{i=1}^{n-1}
		\langle w_i^k,u_{i}^k\rangle 
		-
		\sum_{i=1}^{n}
		\langle x_i^k,y_i^k\rangle  
		$\label{lineComputeHplane}\;
		$\alpha_k = \frac{\beta_k}{\pi_k}\max\left\{0,\varphi_k(p^k)\right\}
		$\label{lineAlphaCompute}\;
	}
	{
	    \If{$\cup_{j=1}^k I_j=\{1,\ldots,n\}$ \label{allprocessed}}
	    {
		    \Return $z^{k+1}\leftarrow x_n^k, w_1^{k+1}\leftarrow y_1^k,\ldots,w_{n-1}^{k+1}\leftarrow y_{n-1}^k$\label{lineReturn}\;
	    }
        \Else
        {
           $\alpha_k = 0$\;
        }
	} 
$z^{k+1} = z^k - \gamma^{-1}\alpha_k v^k$\label{eqAlgproj1} \;
$w_i^{k+1} = w_i^k - \alpha_k u_{i}^k,\quad i=1,\ldots,n-1$,\label{eqAlgproj2} \;
$w_{n}^{k+1} = -\sum_{i=1}^{n-1} G_i^* w_{i}^{k+1}$\label{lineCoordEnd}\;
}
\caption{General Projective Splitting Algorithm for solving~\eqref{probCompMonoMany}.}
\label{AlgfullyAsync}
\end{algorithm}

\label{secAlgOverview}
Algorithm~\ref{AlgfullyAsync} is our \col{flexible} block-iterative projective splitting
algorithm with forward steps for solving~\eqref{probCompMonoMany}. 
It is essentially a special case of
the weakly convergent Algorithm of~\cite{combettes2016async}, except that we use
the new forward-step procedure to deal with the Lipschitz continuous operators
$T_i$ for $i\in\Iforw$, instead of exclusively using proximal steps. For our
separating hyperplane in \eqref{defHyper}, we use a special case of the formulation of \cite{combettes2016async},
which is slightly different from the one used in~\cite{eckstein2017simplified}.
Our method can be reformulated to use the same hyperplane as
\cite{eckstein2017simplified}; however, this requires that it be computationally
feasible to project on the subspace given by the equation $\sum_{i=1}^{n} G_i^*w_i
 = 0$.

\col{ 
Under appropriate conditions, Algorithm \ref{AlgfullyAsync} is an instance of Algorithm \ref{AlgAbstractProject} (see Lemma \ref{lemIsProject}). Lines \ref{lineCoordStart}--\ref{lineCoordEnd} of Algorithm \ref{AlgfullyAsync} essentially implement the projection step on line \ref{eqProjectUpdate} of Algorithm \ref{AlgAbstractProject}. Lines \ref{lineStartFor}--\ref{lineLeave}  construct the points $(x_i^k,y_i^k)$ used to define the affine function $\varphi_k$ in \eqref{defHyper}, which defines the separating hyperplane. 
}


The algorithm has the following parameters:
\begin{itemize}
	\item For each iteration $k\geq 1$, a subset $I_k\subseteq
	\{1,\ldots,n\}$. \col{These are the indices of the  ``active'' operators
	that iteration $k$ processes by either a backward step or two forward
	steps.  The remaining, ``inactive'' operators simply have
	$(x_i^k,y_i^k)=(x_i^{k-1},y_i^{k-1})$.}
	\item For each iteration $k\geq 1$ and $i=1,\ldots,n$, a delayed iteration
	index $d(i,k)\in\{1,\ldots,k\}$ which allows the subproblem calculations
	\col{on lines \ref{lineaupdate}--\ref{ForwardyUpdate} to use outdated
	information $(z^{d(i,k)}, w_i^{d(i,k)})$. In the most straightforward case
	of no delays, $d(i,k)$ is simply $k$.} 
	\item For each $k\geq 1$ and $i=1,\ldots,n$, a positive scalar 
	stepsize $\rho_{i}^{k}$.
	\item For each iteration $k\geq 1$, an overrelaxation parameter
	$\beta_k\in [\underline{\beta},\overline{\beta}]$ for some constants
	$0<\underline{\beta}\leq\overline{\beta}<2$.
	\item A scalar $\gamma>0$ which controls the relative emphasis on the
	primal and dual variables in the projection update in lines
	\ref{eqAlgproj1}-\ref{eqAlgproj2};
	see~\eqref{gammanorm} in Section~\ref{secHplane} for more details. 
	\item Sequences of errors $\{e_i^k\}_{k\geq 1}$ for $i\in\Iback$ \col{modeling}	
	inexact computation of the proximal steps.
\end{itemize}

\col{ 
In the form directly presented in Algorithm~\ref{AlgfullyAsync}, the delay
indices $d(i,k)$ may seem unmotivated; it might seem best to always select
$d(i,k) = k$.  However, these indices can play a critical role in modeling
asynchronous parallel implementation.   
There are many ways in which Algorithm~\ref{AlgfullyAsync} could be
implemented in various parallel computing environments; a specific suggestion
for asynchronous implementation of a closely related class of algorithms is
developed in~\cite[Section 3]{eckstein2017simplified}. 

The error parameters $e_i^k$ for the proximal steps would simply be zero for
proximal steps that are calculated exactly.  When nonzero, they would not
typically in practice be explicitly chosen prior to calculating $x_i^k$ and
$y_i^k$, but instead implicitly defined by some (likely iterative) procedure
for approximating the $\prox$ operation.  We present the error parameters as
shown in order to avoid cluttering the algorithm description with additional
loops and abstractions as
in~\cite{eckstein2017approximate,eckstein2017relative}.}



\col{ 
\subsection{A Block-Iterative Implementation}
Before proceeding with the analysis of Algorithm \ref{AlgfullyAsync}, we
present a somewhat simplified block-iterative version. This version eliminates
the possibility of delays, setting $d(i,k) \equiv k$. The strategy for
deciding which operators $I_k$ to select at each iteration is left open for
the time being and is determined entirely by the algorithm implementer.
We will propose one specific strategy 
for the case $\lvert{I_k}\rvert \equiv1$
in Section \ref{secGreed}, but one may use any approach conforming to
Assumption~\ref{assAsync}(\ref{quasicyclic}) below.}

\begin{algorithm}[h!]
	\DontPrintSemicolon
	\SetKwInOut{Input}{Input}
	\SetKwInOut{Parameters}{Parameters}
\col{ 
	\Input{$(z^1,{\bf w}^1)\in \boldsymbol{\mathcal{H}}$, $(x_i^0,y_i^0)\in\mathcal{H}_i^2$ for $i=1,\ldots,n$.}
	\Parameters{$\{I_k\}_{k\in\nN}$ where $I_k\subseteq\{1,\ldots,n\}$, $0<\underline{\beta}\leq\overline{\beta}<2$, $\gamma>0$.}
	
	\For{$k=1,2,\ldots$}
	{   		
		
			\For{$i\in I_k$}
			{ 
			  \tcc{Loop over the blocks chosen to be updated according to user-supplied rule $\{I_k\}$}
				\If{$i\in\Iback$}
				{
					$a = G_i z^{k}+\rho_{i}^{k} w_i^{k}+e_i^k$\label{lineaupdate2}
					\tcc*{do a backward step}
					$x_i^k = \prox_{\rho_{i}^{k} T_i}(a)
					$\label{LinebackwardUpdate2}\;
					$
					y_i^k = (\rho_{i}^{k})^{-1}
					\left(
					a - x_i^k
					\right)
					$\label{lineBackwardUpdateY2}\;

				}
				\Else
				{\label{LineForwardUpdate2}
					$					
					x_i^k = G_i z^{k}-\rho_{i}^{k}
					( T_i G_i z^{k} - w_{i}^{k}),
					$\label{ForwardxUpdate2}
					\tcc*{do two forward steps}            	
					$
					y_i^k = T_i x_i^k.
					$
					\label{ForwardyUpdate2}\;
					
				}
			}
			
			For $j\notin I_k$, set $(x_j^k,y_j^k)=(x_j^{k-1},y_j^{k-1})$\tcc*{other blocks unchanged}
			\tcc{The projection procedure is then the same as lines \ref{lineCoordStart}-\ref{lineCoordEnd} of Algorithm \ref{AlgfullyAsync}}
		
	}
}
	\caption{\col{Simplified Block-Iterative 
	              Algorithm.}}
	\label{AlgBlockIter}
\end{algorithm}

\section{Convergence Analysis}

We now start our analysis of the weak convergence of the iterates of Algorithm
\ref{AlgfullyAsync} to a solution of problem~(\ref{probCompMonoMany}). While
the overall proof strategy is similar to \cite{eckstein2017simplified}, considerable
innovation is required to incorporate the forward steps.
\col{Before the main proof, we will first state our assumptions on Algorithm \ref{AlgfullyAsync} and its parameters, state the main convergence theorem, and sketch an outline of the proof.

\subsection{Algorithm Assumptions}
We start with our assumptions about parameters of Algorithm
\ref{AlgfullyAsync}.  With the exception of~\eqref{eqForwUpper}, they are
taken from~\cite{combettes2016async,eckstein2017simplified} and use the notation
of~\cite{eckstein2017simplified}. }

\begin{assumption}\label{assAsync}
	For Algorithm \ref{AlgfullyAsync}, assume:
	\begin{enumerate}
		\item \label{quasicyclic} For some fixed integer $M\geq 1$, 
		each index $i$ in $1,\ldots,n$
		is in $I_k$ at least once every $M$ iterations, that is,
		$$
		(\forall\,j\geq 1) \qquad \bigcup\limits_{k=j}^{j+M-1}I_k = \{1,\ldots,n\}.
		$$
		\item For some fixed integer $D\geq 0$, we have $k-d(i,k)\leq D$ for
		all $i, k$ with $i\in I_k$. That is, there is a constant bound on the
		extent to which the information $z^{d(i,k)}$ and $w_i^{d(i,k)}$ used
		in lines~\ref{lineaupdate} and~\ref{ForwardxUpdate} is out of
		date.
	\end{enumerate}
\end{assumption}

\begin{assumption}\label{assStep}
	The stepsize conditions for weak convergence of Algorithm \ref{AlgfullyAsync} are:
	\begin{align}\nonumber
	\underline{\rho} &\triangleq 
	\min_{i=1,\ldots,n} \left\{ \inf_{k\geq 1} \rho_i^k\right\} > 0 &
	\overline{\rho} &\triangleq
	\max_{i\in\Iback} \left\{\sup_{k\geq 1} \rho_i^k \right\} < \infty 
	\end{align}
	\vspace{-3.5ex}
	\begin{align} \label{eqForwUpper}
	(\forall\,i\in\Iforw)\quad
	\overline{\rho}_i &\triangleq 
	\limsup_{k\to\infty} \rho_i^k < \frac{1}{L_i}.
	\end{align}
\end{assumption}
\noindent Note that (\ref{eqForwUpper}) allows the stepsize to be larger than
the right hand side for a finite number of iterations.

\col{The last assumption concerns the possible errors $e_i^k$ in computing the
proximal steps and requires some notation from~\cite{eckstein2017simplified}:
for all $i$ and $k$, define}
\begin{align*}
\nonumber
S(i,k)
&=
\{j\in\mathbb{N}:j\leq k,i\in I_j \}
&
s(i,k)
&=
\left\{
\begin{array}{ll}
\max S(i,k),\quad& \text{when } S(i,k)\neq \emptyset\\
0,&\text{otherwise.}
\end{array}
\right.
\end{align*}
In words, $s(i,k)$ is the most recent iteration up to and including $k$ in
which the index-$i$ information in the separator was updated, or $0$ if
index-$i$ information has never been processed. Assumption \ref{assAsync}
ensures that $0\leq k-s(i,k)< M$.

Next, for all $i=1,\ldots,n$ and iterations $k$, define
$
l(i,k) = d\big(i,s(i,k)\big).
$
Thus, $l(i,k)$ is the iteration in which the algorithm generated the
information $z^{l(i,k)}$ and $w_{i}^{l(i,k)}$ used to compute the current
point $(x_i^k,y_i^k)$. For initialization purposes, we set $d(i,0) = 0$.

\begin{assumption}\label{assErr}
	The error sequences $\{e_i^k\}$ are bounded for all $i\in\Iback$. For some $\sigma$
	with $0\leq\sigma<1$ the following hold for all
	$k \geq 1$:
	\begin{align}
	(\forall i\in\Iback)&&
	\langle G_i z^{l(i,k)} - x_i^k,e_i^{s(i,k)}\rangle
	&\geq
	-\sigma\|G_i z^{l(i,k)} - x_i^k\|^2
	\label{err1}\\
	(\forall i\in \Iback)&&
	\langle e_i^{s(i,k)},y_i^k - w_i^{l(i,k)}\rangle 
	&\leq \rho^{l(i,k)}_i\sigma \|y_i^k - w_i^{l(i,k)}\|^2.
	\label{err2}
	\end{align}
\end{assumption}

\col{ 
\subsection{Main Result}
We now state the main technical result of the paper, asserting weak convergence
of Algorithm \ref{AlgfullyAsync} to a solution of~\eqref{probCompMonoMany}.
\begin{theorem}\label{ThmAsync}
	Suppose Assumptions \ref{AssMonoProb}-\ref{assErr} hold. 
	If Algorithm~\ref{AlgfullyAsync} terminates at line~\ref{lineReturn}, then its 
	final iterate is a member of the extended solution set $\calS$.
	Otherwise, the
	sequence $\{(z^k,\bw^k)\}$ generated by Algorithm~\ref{AlgfullyAsync}
	converges weakly to some point $(\bar z,\overline{\bw})$ in the extended
	solution set $\calS$ of~\eqref{probCompMonoMany} defined
	in~\eqref{defCompExtSol}. Furthermore, $x_i^k\rightharpoonup G_i \bar z$
	and $y_i^k\rightharpoonup \overline{w}_i$ for all $i=1,\ldots,n-1$,
	$x_{n}^k\rightharpoonup \bar z$, and $y_n^k\rightharpoonup
	-\sum_{i=1}^{n-1}G_i^* \overline{w}_i$. 
\end{theorem}

Before establishing this result, we first outline the basic proof strategy:
first, since it arises from a projection method, the sequence $\{p^k\}$ has
many desirable properties, as outlined in Lemma~\ref{lemFejer}. In particular,
Lemma \ref{lemFejer}(\ref{pointFejer}) allows us to establish (weak) convergence of
the entire sequence to a solution if we can prove that all its limit points must be elements
of $\calS$. To that end, we will establish that
\begin{align}\label{eqKey}
(\forall i=1,\ldots,n):\quad  G_i z^k - x_i^k\to 0\text{ and }y_i^k - w_i^k\to 0.
\end{align}
By the definition of $w_n^k$ on line \ref{lineCoordEnd}, the iterates
$(z^k,\bw^k)$ always meet the linear relationship between the $w_i$ 
implicit in the definition~\eqref{defCompExtSol} of~$\calS$, whereas the
$(x_i^k,y_i^k)$ iterates always meet its inclusion conditions.  Therefore,
if~\eqref{eqKey} holds, then one may expect all limit points of $(z^k,\bw^k)$
to satisfy all the conditions in~\eqref{defCompExtSol} and thus to to lie in
$\calS$. In finite dimension, this result is in fact fairly straightforward to
establish.
The general Hilbert space proof is more
delicate, but was carried out in \cite[Proposition 2.4]{alotaibi2014solving}.

In order to establish \eqref{eqKey}, 
we will first establish that the gradient of the affine function $\varphi_k$ defined in \eqref{hyper} remains bounded. Then, consider the projection update as written on line \ref{eqProjectUpdate} of Algorithm \ref{AlgAbstractProject}, which implies
\begin{align*}
\|p^{k+1} - p^k\| =
\frac{\beta_k\max\{0,\varphi_k(p^k)\}}
{\|\nabla\varphi_k\|_{\bcalH}}.
\end{align*}
If $\|\nabla\varphi_k\|_{\bcalH}$ remains bounded, then since Lemma \ref{lemFejer}(\ref{pointDiff20}) implies the left hand side goes to $0$,  $\limsup \varphi_k(p^k)\leq 0$.

The key to establishing \eqref{eqKey} is then to show that the cut provided by the separating hyperplane is ``sufficiently deep". This will amount to proving (in simplified form) 
\begin{align}\label{eqSimpleLow}
\varphi_k(p^k)\geq C\sum_{i=1}^n \|G_i z^k - x_i^k\|^2
\end{align}
for some $C>0$. Then, using $\limsup \varphi_k(p^k)\leq 0$, the first part of
\eqref{eqKey} follows. The second part of \eqref{eqKey} is then established by
a similar argument. }


\subsection{Preliminary Lemmas}\label{secAssBlock}


To begin the proof of Theorem~\ref{ThmAsync}, we first deal with the situation in
which Algorithm \ref{AlgfullyAsync} terminates at line \ref{lineReturn}.

\begin{lemma}  \label{lemFiniteTerm}
	For Algorithm~\ref{AlgfullyAsync}:
	\begin{enumerate}
		\item 
	Suppose Assumption \ref{AssMonoProb} holds. If the algorithm terminates
	 via line \ref{lineReturn}, then $(z^{k+1},{\bf w}^{k+1})\in\calS$.
	 Furthermore $x_i^k = G_i z^{k+1}$ and $y_i^k = w_i^{k+1}$ for
	 $i=1,\ldots,n-1$, and $x_n^k = z^{k+1}$ and $y_n^k =
	 -\sum_{i=1}^{n-1}G_i^* w_i^{k+1}$.
	\item Additionally, 
	suppose Assumption \ref{assAsync}(1) holds. Then if $\pi_k=0$ at some
    iteration $k\geq M$, the algorithm terminates via line~\ref{lineReturn}.
	\end{enumerate}
\end{lemma} 
\begin{proof}
	The condition $\cup_{j=1}^k I_j=\{1,\ldots,n\}$ on line~\ref{allprocessed}
	implies that $y_i^k\in T_i
	x_i^k$ for $i=1,\ldots,n$. Let $\varphi_k$ be the affine function defined in \eqref{defCompHplane}.
	Simple algebra verifies that for $u^k$ and $v^k$ defined on
	lines \ref{lineCoordStart} and \ref{lineVupdate}, $u_{i}^k =
	\nabla_{w_i}\varphi_k$ for $i=1,\ldots,n-1$, $v^k =
	\gamma\nabla_z\varphi_k$, and $\pi_k = \|\nabla\varphi_k\|_\gamma^2$.
	 If for any such $k$, $\pi_k$ equals $0$, then this
	implies $\nabla\varphi_k = 0$.	
	Then we can invoke Lemma~\ref{LemGradAffine}(\ref{lem0grad}) to conclude
	that $(x_n^k,y_1^k,\ldots,y_{n-1}^k)\in \calS$.  Thus, the algorithm
	terminates with 
	$$(z^{k+1},w_1^{k+1},\ldots,w_{n-1}^{k+1})=(x_n^k,y_1^k,
	\ldots,y_{n-1}^k)\in\calS.$$ Furthermore, when $\nabla\varphi_k = 0$, 
	Lemma~\ref{LemGradAffine}(\ref{item:gradForm}) leads to
    \begin{align*}
	\sum_{i=1}^{n-1} G_i^* y_i^k+ y_n^k &= 0 &
	x_i^k - G_i x_{n}^k &= 0 \quad i=1,\ldots,n-1.
    \end{align*}
    We immediately conclude that $y_n^k = -\sum_{i=1}^{n-1} G_i^* y_i^k =
	-\sum_{i=1}^{n-1} G_i^* w_i^{k+1}$ and, for $i=1,\ldots,n-1$, that  
	$x_i^k = G_i x_n^k = G_i z^k$.
	
	Finally, note that for any $k\geq M$,  $\cup_{j=1}^k I_j=\{1,\ldots,n\}$
	by Assumption \ref{assAsync}(1). Therefore whenever $\pi_k = 0$ for $k\geq
	M$, the algorithm terminates via line \ref{lineReturn}.
	\ourqed 
\end{proof}
Lemma 5 asserts that if the algorithm terminates finitely, then the final
iterate is a solution. For the rest of the analysis, we therefore assume that
$\pi_k \neq 0$ for all $k\geq M$. Under Assumption \ref{assAsync},
Algorithm~\ref{AlgfullyAsync} is a projection algorithm:
\begin{lemma}\label{lemIsProject}
	Suppose that Assumption~\ref{assMono} holds for
	problem~\eqref{probCompMonoMany} and~Assumption \ref{assAsync}(1) holds for
	Algorithm~\ref{AlgfullyAsync}. Then, for all $k\geq M$ 
	such that $\pi_k$ defined on Line \ref{linePiUpdate} is nonzero, Algorithm
	\ref{AlgfullyAsync} is an instance of Algorithm~\ref{AlgAbstractProject}
	with $\bcH = \calH_0 \times \cdots \times \calH_{n-1}$ and the inner product
	in~\eqref{gammanorm}, $\calS$ as defined in \eqref{defCompExtSol}, and
	$\varphi_k$ as defined in~\eqref{defHyper}.  All the statements of
	Lemma~\ref{propFejer} hold for the sequence $\{p^k\} =
	\{(z^k,w_1^k,\ldots,w_{n-1}^k)\}$ generated by
	Algorithm~\ref{AlgfullyAsync}.
\end{lemma}
\begin{proof}
For $k\geq M$ in Algorithm \ref{AlgfullyAsync}, by Assumption \ref{assAsync}(1) 
all $(x_i^k,y_i^k)$ have been updated at least once using either lines
\ref{LinebackwardUpdate}--\ref{lineBackwardUpdateY} or lines
\ref{ForwardxUpdate}--\ref{ForwardyUpdate}, and thus $y_i^k\in T_i x_i^k$ for
$i=1,\ldots,n$. Therefore, Lemma~\ref{LemGradAffine} implies that 
$\varphi_k(z,{\bf w})\leq 0$. 
		
Next we verify that lines
\ref{lineCoordStart}-\ref{lineCoordEnd} of Algorithm~\ref{AlgfullyAsync} are an
instantiation of line \ref{eqProjectUpdate} of
Algorithm~\ref{AlgAbstractProject} using $\varphi_k$ as defined
in~\eqref{hplane} and the norm defined in~\eqref{gammanorm}.  
As already shown, $\pi_k = \|\nabla\varphi_k\|_\gamma^2$.
Considering the decomposition of $\varphi_k$ in~\eqref{eqAffineExpress}, it
can then be seen that lines~\ref{linePiUpdate}-\ref{eqAlgproj2} of
Algorithm~\ref{AlgfullyAsync} implement the projection on
line~\ref{eqProjectUpdate} of Algorithm~\ref{AlgAbstractProject}.

To conclude the proof, we note that Lemma~\ref{lemClosed} asserts that $\calS$ is
closed and convex, so all the results of Lemma~\ref{propFejer} apply.
\ourqed 
\end{proof}

The next two lemmas concern the indices $s(i,k)$ and $l(i,k)$ defined in
Section~\ref{assAsync}.

\begin{lemma}
	\label{lemUpdates}
Suppose Assumption \ref{assAsync}(1) holds. For all iterations $k\geq  M$, if Algorithm \ref{AlgfullyAsync} has not already terminated, then the
updates may be written as
\begin{align}
\label{eqBack}
(\forall i\in \Iback)&&
x^k_i+\rho_i^{l(i,k)}y^k_i 
&=
G_i z^{l(i,k)}+\rho^{l(i,k)}_iw_i^{l(i,k)}+e_i^{s(i,k)}, 
\\\nonumber 
&&y^k_i&\in T_i x^k_i,
	\\\label{eqXForw}
(\forall i\in\Iforw)&&
x^k_{i} 
&=
G_{i} z^{l(i,k)} - \rho^{l(i,k)}_{i}(T_{i} G_{i} z^{l(i,k)} - w^{l(i,k)}_{i}),
\\\nonumber 
&&
y^k_{i} 
&=
T_i x^k_{i}.
\end{align}
\end{lemma}
\begin{proof}
The proof follows from the definition of $l(i,k)$ and $s(i,k)$. After $M$
iterations, all operators must have been in $I_k$ at least once. Thus, after
$M$ iterations, every operator has been updated at least once using either the
proximal step on lines~\ref{lineaupdate}-\ref{lineBackwardUpdateY} or the
forward steps on lines~\ref{ForwardxUpdate}-\ref{ForwardyUpdate} of
Algorithm~\ref{AlgfullyAsync}. Recall the variables defined to ease
mathematical presentation, namely $G_{n} = I$ and $w_{n}^k$ defined
in~\eqref{defwn} and line~\ref{lineCoordEnd}.
\ourqed 
\end{proof}

We now derive some important properties of $l(i,k)$. The following result was
proved in Lemma 6 of \cite{eckstein2017simplified} but since it is short we
include the proof here.
\begin{lemma}\label{LemBoundOnl}
Under Assumption \ref{assAsync}, $k-l(i,k) < M+D$ for all $i=1,\ldots,n$ and
iterations $k$.
\end{lemma} 
\begin{proof}
From the definition, we know that $0\leq k-s(i,k) < M$. Part 2 of
Assumption~\ref{assAsync} ensures that $s(i,k)-l(i,k)=s(i,k) - d(i,s(i,k))\leq
D$.  Adding these two inequalities yields the desired result.
\ourqed 
\end{proof}

\begin{lemma}
	\label{lemDif20}
	Suppose Assumptions~\ref{AssMonoProb} and \ref{assAsync} hold and 
	$\pi_k>0$ for all $k\geq M$.  Then
	$w_i^{l(i,k)}-w_i^k\to 0$ for all $i=1,\ldots,n$ and $z^{l(i,k)}-z^k \to
	0$.
\end{lemma}
\begin{proof}
	For  $z^k$ and $w_i^k$ for $i=1,\ldots,n-1$, the proof is identical to the
	proof of~\cite[Lemma~9]{eckstein2017simplified}.   
	For $\{w_n^k\}$, we
	have from line~\ref{lineCoordEnd} of the algorithm that
\begin{align*}
\|w_{n}^{l(n,k)} - w_{n}^k\|
&=
\left\|
\sum_{i=1}^{n-1}
G_i^*\left(
w_i^k - w_i^{l(n,k)}
\right) 
\right\|
\\
&\leq
\sum_{i=1}^{n-1}
\|G_i^*\|
\left\|
w_i^k - w_i^{l(n,k)}
\right\|.
 \\
 &=
 \sum_{i=1}^{n-1}
 \|G_i^*\|
 \left\|
 \sum_{j=1}^{k-l(n,k)}
 \left(
 w_i^{k-j+1} - w_i^{k-j}
 \right)
 \right\|
 \\
 &\leq
 \sum_{i=1}^{n-1}
 \|G_i^*\|
 \sum_{j=1}^{k-l(n,k)}
 \left\|
 w_i^{k-j+1} - w_i^{k-j}
 \right\|
 \\
 &\leq
 \sum_{i=1}^{n-1}
 \|G_i^*\|
 \sum_{j=1}^{M+D}
 \left\|
 w_i^{k-j+1} - w_i^{k-j}
 \right\|,
\end{align*}
where final line uses Lemma \ref{LemBoundOnl}.
Since  the operators $G_i$ are bounded and Lemma~\ref{propFejer}(2) implies
that $w_i^{k+1}-w_i^k \to 0$ for all $i=1,\ldots,n-1$, we conclude that
$w_n^{l(n,k)}-w_n^k\to 0$.
\ourqed 
\end{proof} 
\noindent Next, we define
\begin{align}
(\forall i=1,\ldots,n) \quad
\phi_{ik} &\triangleq \langle G_i z^k - x^k_i,y_i^k - w_i^k\rangle 
&
\phi_k &\triangleq \textstyle{\sum_{i=1}^{n}\phi_{ik}}
\label{defPhi}\\\label{defPsi}
(\forall i=1,\ldots,n) \quad
\psi_{ik} &\triangleq \langle G_i z^{l(i,k)} - x^k_i,y_i^k - w^{l(i,k)}_i\rangle
&
\psi_k &\triangleq \textstyle{\sum_{i=1}^{n}\psi_{ik}}.
\end{align}
Note that~\eqref{defPhi} simply
expands the definition of the affine function in (\ref{hplane}) and we may write $\varphi_k(p^k) = \phi_k$.
\begin{lemma}\label{LemPsiAndPhi}
	Suppose assumptions~\ref{assMono} and \ref{assAsync} hold and $\pi_k>0$ for all $k\geq M$. Then
	$\phi_{ik}-\psi_{ik}\to 0$ for all $i=1,\ldots,n$.
\end{lemma}
\begin{proof}
In view of Lemma~\ref{lemDif20}, we may follow the same argument
as given in \cite[Lemma~12]{eckstein2017simplified}.
\ourqed 
\end{proof}

\subsection{Three Technical Lemmas}
We now prove three technical lemmas which pave the way to establishing weak
convergence of Algorithm~\ref{AlgfullyAsync} to a solution
of~\eqref{probCompMonoMany}. The first lemma upper bounds the norm of the
gradient of $\varphi_k$ at each iteration.

\begin{lemma}\label{boundedGradient}\label{lemBoundedGrad}
	Suppose assumptions~\ref{assMono}-\ref{assErr} hold. Suppose that $\pi_k>0$ for all $k\geq M$. Recall the
	affine function $\varphi_k$ defined in~\eqref{hplane}.  There exists $\xi_1\geq 0$
	such that
$
		\|\nabla\varphi_k\|_\gamma^2
		\leq
\xi_1
$
for all $k \geq 1$. 
\end{lemma}
\begin{proof}
	For $k< M$ the gradient can be trivially bounded by $\max_{1\leq k<  M}
	\|\nabla\varphi_k\|_\gamma^2$.  Now fix any $k \geq M$.  Using
	Lemma~\ref{LemGradAffine},
	\begin{eqnarray}
		\|\nabla \varphi_k\|_\gamma^2
		=
	\gamma^{-1}\left\|\sum_{i=1}^{n-1} G_i^* y^k_i+y_{n}^k\right\|^2
	+
		\sum_{i=1}^{n-1}
		\|x^k_i - G_i x_{n}^k\|^2
    \label{eqGrad}.
	\end{eqnarray}
	Using Lemma~\ref{lemBasic}, we begin by writing the second term on the
	right of~\eqref{eqGrad} as
	\begin{align*}
				\sum_{i=1}^{n-1}
		\|x^k_i - G_i x_{n}^k\|^2
		&\leq
		2\sum_{i=1}^{n-1}
		\left(
		\|x_i^k\|^2 + \|G_i\|^2\|x_{n}^k\|^2
		\right)
		\\
		&\leq
2\sum_{i=1}^{n-1}\|x_i^k\|^2
+
2(n-1)\max_i \left\{\|G_i\|^2\right\}\|x_n^k\|^2.
	\end{align*}
The linear operators $G_i$ are bounded by Assumption \ref{assMono}. We now
check the boundedness of sequences $\{x_i^k\}$, $i=1,\dots,n$. For
$i\in\Iback$, the boundedness of $\{x_i^k\}$ follows from exactly the same argument as
in~\cite[Lemma 10]{eckstein2017simplified}. Now taking any $i\in\Iforw$, we use the
triangle inequality and Lemma~\ref{lemUpdates} to obtain
\begin{align*}
\|x_{i}^k\| 
&\leq
\|G_{i} z^{l(i,k)} - \rho_{i}^{l(i,k)} T_iG_{i} z^{l(i,k)}\|
+
\rho^{l(i,k)}_{i}\|w_{i}^{l(i,k)}\|
\\
&\leq
\|G_{i}\|\|z^{l(i,k)}\|+\rho^{l(i,k)}_{i}\|T_i G_{i}z^{l(i,k)}\|
+\rho^{l(i,k)}_{i}\|w_{i}^{l(i,k)}\|.
\end{align*}
Now the sequences $\{\|z^k\|\}$ and
 $\{\|w_i^k\|\}$ are bounded by Lemma \ref{propFejer}, implying the
 boundedness of $\{\|z^{l(i,k)}\|\}$ and $\{\|w_i^{l(i,k)}\|\}$. Since
 $\{z^{l(i,k)}\}$ is bounded, $G_{i}$ is bounded, and $T_i$ is Lipschitz
 continuous, $\{T_iG_{i}z^{l(i,k)}\}$ is bounded. Finally, the stepsizes
 $\rho_i^k$ are bounded by Assumption~\ref{assStep}. Therefore, $\{x_{i}^k\}$
 is bounded for $i\in\Iforw$, and we may conclude that the second term
 in~\eqref{eqGrad} is bounded.

We next consider the first term in~\eqref{eqGrad}. Rearranging the update
equations for Algorithm~\ref{AlgfullyAsync} as given in Lemma
\ref{lemUpdates}, we may write
\pagebreak[3]
	\begin{align}\label{eqRearrange1}
		y_i^k &= \left(\rho_{i}^{l(i,k)}\right)^{-1}
		\left(G_i z^{l(i,k)} - x_i^k+\rho_{i}^{l(i,k)}w_i^{l(i,k)}+e_i^{s(i,k)}\right),
		& i &\in\Iback
		\\\label{eqRearrange2}
		T_i G_{i} z^{l(i,k)} &=
		\left(\rho_{i}^{l(i,k)}\right)^{-1}
		\left(G_i z^{l(i,k)} - x_i^k+\rho_{i}^{l(i,k)}w_i^{l(i,k)}\right),
		& i&\in\Iforw.
	\end{align}
	Using $G_n = I$, the squared norm in the first term of~\eqref{eqGrad} may be written as
	\begin{align}
    \left\|\sum_{i=1}^{n} G_i^* y^k_i\right\|^2
	&=
	\left\|\sum_{i\in\Iback} G_i^* y^k_i + \sum_{i\in\Iforw} G_{i}^* \left(T_i G_{i}z^{l(i,k)} + y^k_{i} - T_i G_{i}z^{l(i,k)}\right)\right\|^2
	\nonumber\\\nonumber
	&\overset{\text{(a)}}{\leq}
	2\left\|\sum_{i\in\Iback} G_i^* y^k_i + \sum_{i\in\Iforw} G_{i}^* T_i G_{i} z^{l(i,k)}\right\|^2
	\\\nonumber 
	&\qquad 
	+
	2\left\|\sum_{i\in\Iforw} G_{i}^*\left(y^k_{i} - T_iG_{i} z^{l(i,k)}\right)\right\|^2
	\\
	&\overset{\text{(b)}}{\leq}
	4\left\|\sum_{i=1}^n \big(\rho_i^{l(i,k)}\big)^{-1} G_i^* 
	\left( G_i z^{l(i,k)} - x_i^k+\rho_i^{l(i,k)}w_i^{l(i,k)}
	\right)\right\|^2
	\nonumber\\&\qquad
	+\;2|\Iforw|\sum_{i\in\Iforw} \|G_{i}\|^2
	\left\|
	T_i x^k_{i} - T_iG_{i} z^{l(i,k)}\right\|^2	
	\nonumber\\&\qquad 
	+
	4\left\|\sum_{i\in\Iback}\big(\rho_i^{l(i,k)}\big)^{-1} G_i^*e_i^{s(i,k)}\right\|^2
	\label{eqEndBounded}\\
	&\overset{\text{(c)}}{\leq}
	4n\underline{\rho}^{-2}\max_i\big\{\|G_i\|\big\}^2 
	\left(
	\sum_{i=1}^n \left\| 
	 G_i z^{l(i,k)} - x_i^k+\rho_i^{l(i,k)}w_i^{l(i,k)}
	\right\|^2
	\right. 
	\nonumber\\&\qquad\qquad\qquad\qquad\qquad\qquad  
	\left. 
	+
	\sum_{i\in\Iback}
	\|e_i^{s(i,k)}\|^2
	\right)
	\nonumber\\&\qquad
	+\;
	2|\Iforw|\sum_{i\in\Iforw} \|G_{i}\|^2 L_i^2\|x_{i}^k-G_{i}z^{l(i,k)}\|^2 	
	\label{lastBound1}
	\end{align}
	In the above, (a) uses Lemma~\ref{lemBasic}, while (b) is obtained by substituting 
	\eqref{eqRearrange1}-\eqref{eqRearrange2} into the first squared norm
	and using $y_i^k =T_i x_i^k$ for $i\in\Iforw$ in the second, and then using Lemma~\ref{lemBasic}
	on both terms.  Finally, (c) uses Lemma~\ref{lemBasic}, the Lipschitz
	continuity of $T_i$, and Assumption \ref{assStep}.  For each
	$i=1,\ldots,n$, we have that $G_i$ is a bounded operator, the sequences
	$\{z^{l(i,k)}\}$, $\{x_i^k\}$, and $\{w_i^{l(i,k)}\}$ are already known to
	be bounded,
	$\{\rho_i^{l(i,k)}\}$ is bounded by Assumption~\ref{assStep}, and for $i\in\Iback$, $\{e_i^{s(i,k)}\}$ is bounded by Asssumption~\ref{assErr}. We conclude
	that the right hand side of~\eqref{lastBound1} is bounded. Therefore, the
	first term in~\eqref{eqGrad} is bounded and the sequence
	$\{\nabla\varphi_k\}$ must be bounded.
	\ourqed 
\end{proof}	

The second technical lemma establishes a lower bound for the affine
function $\varphi_k$ evaluated at the current point \col{which is similar to \eqref{eqSimpleLow}. This shows that the cut provided by the hyperplane is ``deep enough"} to guarantee weak convergence of the method. 
The lower bound applies to the quantity $\psi_k$ defined in \eqref{defPsi}:
this quantity is easier to analyze than $\phi_k$ and Lemma~\ref{LemPsiAndPhi}
asserts that the difference between the two converges to zero.

\begin{lemma}
\label{PositivePhi}	\label{lemLowerBoundPhi}
	Suppose that assumptions~\ref{AssMonoProb}-\ref{assErr} hold. Suppose $\pi_k>0$ for all $k\geq M$.  
	Then there exists $\xi_2>0$ such that  
\begin{equation}\nonumber
	\limsup_{k\to \infty}\psi_k
	\geq
	\xi_2
	\limsup_{k\to \infty}\sum_{i=1}^n
	\| G_i z^{l(i,k)}-x^k_i\|^2
.
\end{equation}
\end{lemma}
\begin{proof}
For $k\geq  M$, we have
\begin{align}
	\psi_k
	&=
\sum_{i=1}^n
	\left\langle G_i z^{l(i,k)}-x^k_i,y_i^k-w^{l(i,k)}_i\right\rangle
	\nonumber\\ 
	&\overset{\text{(a)}}{=}
	\sum_{i\in\Iback}
	\left\langle G_i z^{l(i,k)}-x^k_i,
	             \big(\rho_i^{l(i,k)}\big)^{-1}\left(G_i z^{l(i,k)}-x^k_i+e_i^{s(i,k)}\right)
	\right\rangle
	\nonumber\\
	&\qquad
	+
	\sum_{i\in\Iforw}
	\left\langle G_{i}z^{l(i,k)}-x^k_{i},T_i G_{i}z^{l(i,k)}-w^{l(i,k)}_{i}
	\right\rangle
	\nonumber\\
	&\qquad
	+
	\sum_{i\in\Iforw}
	\left\langle G_{i}z^{l(i,k)}-x^k_{i},y^k_{i}-T_i G_{i}z^{l(i,k)}
	\right\rangle
	\nonumber\\
	&\overset{\text{(b)}}{=}
	\sum_{i\in\Iback}\left[\big(\rho^{l(i,k)}_i\big)^{-1}\|G_i z^{l(i,k)}-x^k_i\|^2 
	   +\big(\rho^{l(i,k)}_i\big)^{-1}
	\left\langle G_i z^{l(i,k)}-x^k_i,e_i^{s(i,k)}\right\rangle\right]
	\nonumber\\
	&\qquad
	+
	\sum_{i\in\Iforw}
	\left\langle G_{i} z^{l(i,k)}-x^k_{i},
	\big(\rho^{l(i,k)}_{i}\big)^{-1}
	\left(G_{i}z^{l(i,k)}-x^k_{i}\right)\right\rangle 	
	\nonumber\\
	&\qquad
	-
	\sum_{i\in\Iforw}
	\left\langle G_{i}z^{l(i,k)}-x^k_{i},T_i G_{i}z^{l(i,k)} 
	   - T_i x^k_{i}\right\rangle
    \nonumber\\
    &\overset{\text{(c)}}{\geq}
	(1-\sigma)\sum_{i\in\Iback}\big(\rho^{l(i,k)}_i\big)^{-1}\|G_i z^{l(i,k)}-x^k_i\|^2
	\nonumber\\
	&\qquad
	+
	\sum_{i\in\Iforw}
	\left(
	\big(\rho^{l(i,k)}_{i}\big)^{-1} - L_i 
	\right)
	\|G_{i} z^{l(i,k)}-x^k_{i}\|^2.  \label{eqForLineS}
\end{align}
In the above derivation, (a) follows by substitution of~\eqref{eqBack} into
the $\Iback$ terms and algebraic manipulation of the $\Iforw$ terms.  Next,
(b) follows by algebraic manipulation of the $\Iback$ terms and
substitution of~\eqref{eqXForw} into the $\Iforw$ terms.
Finally, (c) is justified by using~\eqref{err1}
in Assumption~\ref{assErr} and the Lipschitz
continuity of $T_i$ for $i\in\Iforw$.

Now consider any two sequences $\{a_k\}\subset\real, \{b_k\} \subset \real_+$.  We note that
\begin{eqnarray*}
\limsup_{k\to\infty} a_k b_k
\geq
\limsup_{k\to\infty}\left\{
\left(
\liminf_{k\to\infty} a_k
\right)
b_k
\right\}
=
\left(\liminf_{k\to\infty} a_k\right)
\left(\limsup_{k\to\infty}
b_k \right).
\end{eqnarray*}
Applying this fact to the expression
in~\eqref{eqForLineS} yields the desired result with
\begin{eqnarray}\label{defXi2}\nonumber
\xi_2 = 
\min
\left\{
(1-\sigma)\overline{\rho}^{-1}
,\;
\min_{j\in\Iforw}
\left\{
\overline{\rho}_{j}^{-1} - L_j 
\right\}
\right\},
\end{eqnarray}
and Assumption~\ref{assStep} guarantees that $\xi_2>0$. 
\ourqed	
\end{proof}
\noindent In the third technical lemma, we provide what is essentially a complementary  
lower bound for $\psi_k$: 
\begin{lemma}
\label{lemLBPhi2}
Suppose assumptions~\ref{AssMonoProb}-\ref{assErr} hold. Suppose $\pi_k>0$ for all $k\geq M$. Then there exists
$\xi_3>0$ such that
\begin{multline}
\limsup_{k\to\infty}\left(
\psi_k 
+ \sum_{i\in\Iforw}L_i \|G_{i} z^{l(i,k)}-x_i^k\|^2
\right)
\\
\geq
\xi_3\limsup_{k\to\infty}
\left(
\sum_{i\in\Iback} \|y_i^k - w_i^{l(i,k)}\|^2
+
\sum_{i\in\Iforw}
\|T_i G_{i} z^{l(i,k)} - w_i^{l(i,k)}\|^2
\right)
.  \label{eqPhiLB2}
\end{multline}
\end{lemma}
\begin{proof}
For all $k\geq  M$, we have 
\pagebreak[3]
\begin{align}
\nonumber
\psi_k &= \sum_{i=1}^n\langle G_i z^{l(i,k)}-x^k_i,y^k_i-w^{l(i,k)}_i\rangle \\
&\overset{\text{(a)}}{=}
\sum_{i\in\Iback}
\langle \rho^{l(i,k)}_i(y^k_i - w^{l(i,k)}_i) - e_i^{s(i,k)},y^k_i - w^{l(i,k)}_i\rangle
\nonumber\\&\qquad
+
\sum_{i\in\Iforw}\langle G_{i} z^{l(i,k)}-x^k_{i},T_i G_{i} z^{l(i,k)}-w_{i}^{l(i,k)}\rangle
\nonumber
\\\nonumber 
&\qquad 
+
\sum_{i\in\Iforw}\langle G_{i}z^{l(i,k)}-x^k_{i},y^k_{i}-T_i G_{i} z^{l(i,k)}\rangle
\\\nonumber&\overset{\text{(b)}}{=}
\sum_{i\in\Iback}
\left(
\rho^{l(i,k)}_i\|y^k_i - w^{l(i,k)}_i\|^2 - \langle e_i^{s(i,k)},y^k_i - w^{l(i,k)}_i\rangle
\right)
\\\nonumber
&\qquad
+
\sum_{i\in\Iforw}\langle \rho^{l(i,k)}_{i}(T_iG_{i} z^{l(i,k)}-w_{i}^{l(i,k)}),T_i G_{i} z^{l(i,k)}-w_{i}^{l(i,k)}\rangle
\\\label{eqPassLemFin}
&\qquad
-
\sum_{i\in\Iforw}\langle x^k_{i}-G_{i}z^{l(i,k)},T_i x^k_{i}-T_i G_{i}z^{l(i,k)}\rangle
\\\nonumber
&\overset{\text{(c)}}{\geq}
(1-\sigma)\sum_{i\in\Iback}\rho^{l(i,k)}_i\|y^k_i - w^{l(i,k)}_i\|^2
+
\sum_{i\in\Iforw}\rho^{l(i,k)}_{i}\|T_i G_{i}z^{l(i,k)}-w_{i}^{l(i,k)}\|^2
\\\label{eqFIna}
&\qquad
-\sum_{i\in\Iforw}L_i\|G_{i}z^{l(i,k)}-x^k_{i}\|^2.
\end{align}
In the above derivation, (a) follows by substition of~\eqref{eqBack} into the
$\Iback$ terms and algebraic manipulation of the $\Iforw$ terms.  Next (b) is
obtained by algebraic simplification of the $\Iback$ terms and substitution
of~\eqref{eqXForw} into the two groups of $\Iforw$ terms. Finally, (c) is
obtained by substituting the error criterion~\eqref{err2} from
Assumption~\ref{assErr} for the $\Iback$ terms and using the Lipschitz
continuity of $T_i$ for the $\Iforw$ terms. Adding the last term
in~\eqref{eqFIna} to both sides yields
\begin{multline*}
\psi_k + \sum_{i\in\Iforw}L_i\|G_{i}z^{l(i,k)}-x^k_{i}\|^2 \\
\geq
(1-\sigma)\sum_{i\in\Iback}\rho^{l(i,k)}_i\|y^k_i - w^{l(i,k)}_i\|^2
+
\sum_{i\in\Iforw}\rho^{l(i,k)}_{i}\|T_i G_{i}z^{l(i,k)}-w_{i}^{l(i,k)}\|^2.
\end{multline*}
Assumption~\ref{assErr} requires that $\sigma<1$ and Assumption~\ref{assStep}
requires that $\rho_i^k\geq \underline{\rho}>0$ for all $i$, so taking limits
in the above inequality implies that~\eqref{eqPhiLB2} holds with
$
\xi_3 = 
(1-\sigma)\underline{\rho}.
$
\ourqed 
\end{proof}
\col{ 
\subsection{Proof of Theorem \ref{ThmAsync}}\label{SecWeakConv}
We are now in a position to complete the proof. 
}
The assertion regarding termination at line~\ref{lineReturn} follows
	immediately from Lemma~\ref{lemFiniteTerm}.  For the remainder of the
	proof, we therefore consider only the case that the algorithm runs
	indefinitely and thus that  $\pi_k > 0$ for all $k \geq M$.
	
		The proof has three parts. The first part establishes that $G_i
	z^k - x_i^k\to 0$ for all $i$ and the second part proves that $y_i^k -
	w_i^k\to 0$ for all $i$. Finally, the third part uses these results in
	conjunction with a result in~\cite{alotaibi2014solving} to show that any
	convergent subsequence of $\{p^k\} = \{(z^k,{\bf w}^k)\}$ generated by the
	algorithm must converge to a point in $\calS$, after which we may simply
	invoke Lemma~\ref{propFejer}.
	
	\vspace{1ex} \noindent 
	\underline{Part 1. Convergence of $G_i z^k - x_i^k\to 0$}
	
	\vspace{1ex}
	\noindent 
	Lemma~\ref{lemIsProject} and~\eqref{defPhi} imply that 
	\[
	p^{k+1} = p^k - \frac{\beta_k\max\{\varphi_k(p^k),0\}}
	                     {\|\nabla\varphi_k\|_\gamma^2}\nabla\varphi_k
	        = p^k - \frac{\beta_k\max\{\phi_k,0\}}
	                     {\|\nabla\varphi_k\|_\gamma^2}\nabla\varphi_k. 
	\]
	Lemma~\ref{propFejer}(2) guarantees that $p^k - p^{k+1} \to 0$, so it follows that
	\[
	0 = \lim_{k\to\infty}\|p^{k+1}-p^k\|_\gamma 
	  = \lim_{k\to\infty}\frac{\beta_k\max\{\phi_k,0\}}{\|\nabla\varphi_k\|_\gamma}
	  \geq\frac{\underline{\beta}\limsup_{k\to\infty}\max\{\phi_k,0\}}{\sqrt{\xi_1}},
	\]
	since $\|\nabla\varphi_k\|_\gamma\leq \sqrt{\xi_1}<\infty$ for all $k$ by
	Lemma \ref{lemBoundedGrad}.  Thus, $\limsup_{k\to\infty} \phi_k\leq
	0$. Since Lemma \ref{LemPsiAndPhi} implies that  $\phi_k -
	\psi_k\to 0$, it follows that $\limsup_{k\to\infty}\psi_k\leq 0$. With (a)
	following from Lemma~\ref{PositivePhi}, we next obtain
	\begin{align*}
		0\geq \limsup_{k\to\infty} \psi_k
		&\overset{\text{(a)}}{\geq}
		\xi_2\lim\sup_k \sum_{i=1}^{n}\|G_i z^{l(i,k)} - x^k_i\|^2
		\\
		&\geq 
		\xi_2\lim\inf_k \sum_{i=1}^{n}\|G_i z^{l(i,k)} - x^k_i\|^2
		\geq 0.
	\end{align*}
Thus,
	$
	G_i z^{l(i,k)}-x_i^k\to 0
	$
	for $i=1,\ldots,n$.
	Since $z^k - z^{l(i,k)}\to 0$ and $G_i$ is bounded, we obtain that 
	$
    G_iz^k - x_i^k\to 0
	$
	for $i=1,\ldots,n$.
	
	\pagebreak[2]
	\noindent\underline{Part 2. Convergence of $y_i^k - w_i^k\to 0$}
	
	\vspace{1ex}
	\noindent 
From $\limsup_{k\to\infty} \psi_k
\leq 0$ and $G_iz^{l(i,k)} - x_i^k\to 0$, we obtain
\begin{eqnarray}
\limsup_{k\to\infty}
\left\{
\psi_k
+	
\sum_{i\in\Iforw} L_i\|G_{i} z^{l(i,k)}-x^k_{i}\|^2
\right\} \leq 0.\label{eq2use}
\end{eqnarray}
Combining (\ref{eq2use}) with (\ref{eqPhiLB2}) in Lemma \ref{lemLBPhi2}, we infer that 
	\begin{align}
	\text{~}&& (\forall\,i\in\Iback)&&
	y_i^k - w_i^{l(i,k)}&\to 0 &\implies&& 
	y_i^k - w_i^{k}&\to 0 &\text{~}& 
	\nonumber
	\\ 
	&&(\forall\,i\in\Iforw)&&
	T_i G_{i} z^{l(i,k)}-w_{i}^{l(i,k)}& \to 0 &\implies&& 
	T_i G_{i} z^k-w_{i}^{k}&\to 0. \label{eqUse}
	\end{align}
where the implications follow from Lemma~\ref{lemDif20}, the Lipschitz
continuity of $T_i$ for $i\in\Iforw$, and the continuity of the linear
operators $G_{i}$. Finally, for each $i\in\Iforw$ and $k\geq M$, we further reason that
	\begin{align}
		\|y^k_{i} - w^k_{i}\|
		&=
		\|T_i G_{i}z^k - w^k_{i}+y^k_{i}-T_i G_{i}z^k\|		\nonumber,\\
		&\leq
		\|T_i G_{i}z^k - w^k_{i}\|+\|y^k_{i}-T_i G_{i}z^k\|   \nonumber	\\
		&\overset{\text{(a)}}{=}
		\|T_iG_{i}z^k - w^k_{i}\|+\|T_i x^k_{i}-T_i G_{i}z^k\|
		\nonumber\\
		\nonumber
        &\overset{\text{(b)}}{\leq}
		\|T_i G_{i}z^k - w^k_{i}\|+L_i\|G_{i}z^k-x^k_{i}\|
		   \overset{\text{(c)}}{\to} 0.
	\end{align}
Here, (a) uses~\eqref{eqXForw} from Lemma~\ref{lemUpdates}, (b) uses the
Lipschitz continuity of $T_i$, and (c) relies on~\eqref{eqUse} and part 1 of this proof.

	\vspace{1ex} \noindent
	\underline{Part 3. Subsequential convergence}
	
	\vspace{1ex}
	\noindent 
	Consider any increasing sequence of indices $\{q_k\}$ such that
	$(z^{q_k},{\bf w}^{q_k})$ weakly converges to some point
	$(z^\infty,\bw^\infty) \in \bcH$.  We claim that in
	any such situation, $(z^\infty,{\bf w}^\infty)\in \calS$.

	By part 1, $z^k - x^k_{n}\to 0$, so $x_{n}^{q_k}\rightharpoonup z^\infty$.
	For any $i=1,\ldots, n$, part 2 asserts that $y_i^k-w^k_i\to 0$, so 
	$y^{q_k}_i\rightharpoonup w^\infty_i$. Furthermore, 
	part 2,  \eqref{defwn}, and the boundedness of $G_i$ imply that
\[
\sum_{i=1}^{n}G_i^* y_i^k
=
\sum_{i=1}^{n}
G_i^*w_i^k
   +
\sum_{i=1}^{n}
G_i^*(y_i^k - w_i^k)
\to 0.
\]
Finally, part 1 and the boundedness of $G_i$ yield
\[
(\forall\,i=1,\ldots,n-1)  \quad\quad
x_i^k - G_i x_{n}^k 
= 
x_i^k - G_i z^k - G_i(x_{n}^k - z^k)
\to 0.
\]
Next we apply~\cite[Proposition 2.4]{alotaibi2014solving} with the following
change of notation where ``MM'' stands for ``maximal monotone'' and ``BL''
stands for ``bounded linear'':
	\begin{eqnarray*}
	    \text{\textbf{Notation here}} & & 
	    \text{\textbf{Notation in \cite{alotaibi2014solving}}} \\
		\text{iteration counter }k &\longrightarrow& \text{iteration counter }n
		\\
		x_{n}^k 
		&\longrightarrow & 
		a_n
		\\
		(x_1^k,\ldots ,x_{n-1}^k) 
		&\longrightarrow & 
		b_n
		\\
		y_{n}^k
		&\longrightarrow&
		a_n^*
		\\
		(y_1^k,\ldots,y_{n-1}^k)
		&\longrightarrow&
		b_n^*
		\\
		T_n 
	    &\longrightarrow &
	    A\text{ (MM operator)}
	    \\
		(x_1,\ldots ,x_{n-1}) \mapsto T_1 x_1 \times \cdots \times T_{n-1} x_{n-1}
		&\longrightarrow &
		B\text{ (MM operator)}
		\\
		z \mapsto (G_1 z,\ldots,G_{n-1} z)
		&\longrightarrow&
		L\text{ (BL operator)}
		\\
		z^\infty &\longrightarrow & \bar{x}
		\\
		{\bf w}^\infty &\longrightarrow & \bar{v}^*.
	\end{eqnarray*}
	We then conclude from~\cite[Proposition 2.4]{alotaibi2014solving} that
	$(z^\infty,{\bf w}^\infty)\in \calS$, and the claim is established.  

	Invoking Lemma~\ref{propFejer}(3), we immediately conclude that
	$\{(z^k,\bw^k)\}$ converges weakly to some $(\bar z,\overline{\bw}) \in
	\calS$.  For each $i=1,\ldots,n$, we finally observe that 
	since $G_i z^k - x_i^k\to 0$ and $y_i^k-w_i^k\to 0$, we also have
	$x_i^k\rightharpoonup G_i \bar z$ and $y_i^k\rightharpoonup \overline{w}_i$.
	\ourqed

\section{Extensions}
\subsection{Backtracking Linesearch}
\label{secLineSearch}
This section describes a backtracking linesearch procedure that may be
used in the forward steps when the Lipschitz constant is unknown. The
backtracking procedure is formalized in Algorithm \ref{AlgBackTrack}, to be
used in place of lines \ref{ForwardxUpdate}-\ref{ForwardyUpdate} of Algorithm \ref{AlgfullyAsync}. 

\begin{algorithm}[h]\label{AlgBackTrack}
	\DontPrintSemicolon
	\SetKwInOut{Input}{Input}
	\SetKwInOut{Output}{Output}
	\caption{Backtracking procedure for unknown Lipschitz constants}
	\label{AlgLineSearch}
	\Input{$i,k,z^{d(i,k)}$, $w_i^{d(i,k)}$, $\rho_i^{d(i,k)}$, $\Delta$}
	$\rho_i^{(1,k)} = \rho_i^{d(i,k)}$\;
	$\theta_i^k = G_{i}z^{d(i,k)}$\;
	$\zeta_i^k = T_i \theta_i^k$\;
	\For{$j=1,2,\ldots$}{
	   $\tilde{x}_i^{(j,k)} 
	      = \theta_i^k - \rho_i^{(j,k)}(\zeta_i^k - w_i^{d(i,k)})$ \label{lineX}\;
       $\tilde{y}^{(j,k)}_i=T_i\tilde{x}^{(j,k)}_i$ \label{lineY}\;
       \If{\label{lineIf}$\Delta\|\theta_i^k-\tilde{x}_i^{(j,k)}\|^2 -
           \langle \theta_i^k-\tilde{x}_i^{(j,k)},\tilde{y}_i^{(j,k)}-w_i^{d(i,k)}\rangle
               \leq 0$}{
             \Return{ $J(i,k) \leftarrow j, \;
                       \hat{\rho}_i^{d(i,k)} \leftarrow \rho_i^{(j,k)}, \;
                       x_i^k \leftarrow \tilde{x}_i^{{(j,k)}}, \;
                       y_i^k \leftarrow \tilde{y}_i^{{(j,k)}}$}
                       \label{lineBTreturn}
           }
       $\rho_i^{(j+1,k)} = \rho_i^{(j,k)}/2$\;
	}
\end{algorithm}


We introduce the following notation: as suggested in line~\ref{lineBTreturn} of Algorithm \ref{AlgBackTrack},
we set $J(i,k)$ to be the number of iterations of the backtracking algorithm
for operator $i\in\Iforw$ at outer iteration $k\geq 1$; the subsequent theorem
will show that $J(i,k)$ can be upper bounded.  As also suggested in
line~\ref{lineBTreturn}, we let $\hat{\rho}_i^{d(i,k)} = \rho_i^{(J(i,k),k)}$
for $i \in \Iforw \cap \I_k$. When using the backtracking procedure for
$i\in\Iforw$, it is important to note that the interpretation of
$\rho_i^{d(i,k)}$ changes: it is the \emph{initial} trial stepsize value for the
$i^{\text{th}}$ operator at iteration $k$, and the actual stepsize used is
$\hat{\rho}_i^{d(i,k)}$.
When $i\notin I_k$, we  set $J(i,k) = 0$ and
$\hat{\rho}_i^{d(i,k)} = \rho_i^{d(i,k)}$.

\begin{assumption}\label{assStep2}
Lines \ref{ForwardxUpdate}-\ref{ForwardyUpdate} of
Algorithm~\ref{AlgfullyAsync} are replaced with the procedure in
Algorithm~\ref{AlgBackTrack}.  Regarding stepsizes, we assume that
\begin{align}
	\label{steps2_2}	
	 \overline{\rho} \triangleq \max_{i=1,\ldots,n}\left\{\sup_k\rho_i^k\right\}< \infty
\end{align}
and either:
\begin{align} \label{eqstepsLow1}
\underline{\rho} = \min_{i=1,\ldots,n}\left\{\inf_k\rho_i^k\right\}>0.
\end{align}
\col{or
\begin{align}\label{eqstepsLow2}
\rho_i^{d(i,1)}>0\quad\text{and}\quad(\forall k\geq 2):\quad \rho_i^{d(i,k)}\geq \hat{\rho}_i^{{d(i,k-1)}}.
\end{align}
}
\end{assumption}
\col{In words, \eqref{eqstepsLow2} allows us to initialize the linesearch with a stepsize which is at least as large as the previously discovered stepsize, which is a common procedure in practice. }
\begin{theorem}
	\label{thmBackTrack}
Suppose that Assumptions \ref{AssMonoProb}, \ref{assAsync}, \ref{assErr}, and
\ref{assStep2} hold. Then all the conclusions of Theorem \ref{ThmAsync}
follow. Specifically, either the algorithm terminates in a finite number of
iterations at point in $\calS$, or there exists $(\bar z,\overline{\bf
w})\in\calS$ such that $(z^k,{\bf w}^k)\rightharpoonup (\bar z,\overline{\bf w})$,
$x_i^k\rightharpoonup G_i \bar z$ and $y_i^k\rightharpoonup \overline{w}_i$
for all $i=1,\ldots,n-1$, $x_{n}^k\rightharpoonup \bar z$, and
$y_n^k\rightharpoonup -\sum_{i=1}^{n-1}G_i^*
\overline{w}_i$,
\end{theorem}

\begin{proof}
	The proof of finite termination at an optimal point follows as before, via Lemma \ref{lemFiniteTerm}. From now on, suppose $\pi_k>0$ for all $k\geq M$ implying that the algorithm runs indefinitely. 
	
The proof proceeds along the following outline: first, we upper bound the
number of iterations of the loop in Algorithm~\ref{AlgBackTrack},  implying
that the stepsizes $\hat{\rho}_i^{d(i,k)}$ are bounded from above and below. We then
argue that lemmas~\ref{lemIsProject}-\ref{LemPsiAndPhi} hold as before. Then
we show that lemmas~\ref{lemBoundedGrad}-\ref{lemLBPhi2} essentially still
hold, but with different constants. The rest of the proof then proceeds 
identically to that of Theorem~\ref{ThmAsync}.

Regarding upper bounding the inner loop iterations, fix any $i\in\Iforw$. For
any $k \geq 1$ such that $i\in I_k$ and for any $j\geq 1$, substituting the values
just assigned to $\theta_i^k$ and $\zeta_i^k$ allows us to expand the forward step on
line~\ref{lineX} of Algorithm~\ref{AlgBackTrack} into
\begin{align*}
\tilde{x}_i^{(j,k)} = G_i z^{d(i,k)} - \rho_i^{(j,k)}(T_i G_i z^{d(i,k)} - w_i^{d(i,k)}).
\end{align*}
Following the arguments used to derive the $\Iforw$ terms in~\eqref{eqForLineS}, we have
\begin{equation}\label{eqKey4backTrack}
	\big(\big(\rho_{i}^{(j,k)}\big)^{-1}-L_i\big)
	   \|G_{i} z^{d(i,k)} - \tilde{x}_{i}^{(j,k)}\|^2
	-
	\langle G_{i} z^{d(i,k)} - \tilde{x}_{i}^{(j,k)},\tilde{y}_{i}^{(j,k)} - w_{i}^{d(i,k)}\rangle 
	\leq 
	0.
\end{equation}
Using that $\rho_i^{(j,k)} = 2^{1-j}\rho_i^{d(i,k)}$, some elementary algebraic
manipulations establish that once
\[
j \geq
\left\lceil 
1 + \log_2\!\left((\Delta+L_i)\rho^{d(i,k)}_{i}\right)
\right\rceil,
\]
one must have
$
\Delta \leq 
	\big(\rho_i^{(j,k)}\big)^{-1}-L_i,
$
and by~\eqref{eqKey4backTrack} the condition triggering the return statement
in Algorithm~\ref{AlgBackTrack} must be true.  Therefore, for any
$k\geq 1$ we have
\begin{align} 
J(i,k) 
&\leq
\max\left\{
\left\lceil 
1 + \log_2\!\left((\Delta+L_i)\rho^{d(i,k)}_{i}\right)
\right\rceil
,1\right\}
\nonumber\\\label{JikBound}
&\leq
\max\left\{
2 + \log_2\!\left((\Delta+L_i)\rho^{d(i,k)}_{i}\right)
,
1
\right\}.
\end{align}
By the condition $\overline{\rho} <
\infty$ in~\eqref{steps2_2}, we may now infer that $\{J(i,k)\}_{k\in\nN}$ is
bounded. 
Furthermore, by substituting~\eqref{JikBound}
into $\hat\rho_i^{d(i,k)} = 2^{1-J(i,k)}\rho_i^{d(i,k)}$, we may infer 
for all $k\geq 1$ that
\begin{eqnarray}\label{eqstart}
\hat{\rho}_i^{d(i,k)}\geq \min\left\{\frac{1}{2(L_i+\Delta)},\rho_i^{d(i,k)}\right\}.
\end{eqnarray}
If \eqref{eqstepsLow1} is enforced, then
\begin{align}
\hat{\rho}_i^{d(i,k)}\geq\min\left\{\frac{1}{2(L_i+\Delta)},\rho_i^{d(i,k)}\right\}
\geq \min\left\{\frac{1}{2(L_i+\Delta)},\underline{\rho}\right\}>0.\label{eqRhoHatLB}
\end{align}
\col{On the other hand, if~\eqref{eqstepsLow2} is enforced, then for all $k$
such that $i\in\mathcal{I}_k$, we have
\begin{align}\label{eq2recurse}
\rho_i^{d(i,k+1)}\geq \hat{\rho}_i^{d(i,k)}
\geq \min\left\{\frac{1}{2(L_i+\Delta)},\rho_i^{d(i,k)}\right\}
\end{align}
If $k\notin\calI_k$ then $\hat{\rho}_i^{d(i,k)} = \rho_i^{d(i,k)}$ and $\rho_i^{d(i,k+1)}\geq \rho_i^{d(i,k)}$. Therefore, we may recurse \eqref{eq2recurse} to yield
\begin{align}
\label{eqRhoHatLB2}
\hat{\rho}_i^{d(i,k)}
\geq 
\min\left\{\frac{1}{2(L_i+\Delta)},\rho_i^{d(i,1)}\right\}>0.
\end{align}
}
Finally
since $\hat{\rho}_i^{d(i,k)}\leq \rho_i^{d(i,k)}\leq\bar{\rho}$ for all $k \geq 1$, we must have
$$
\limsup_{k\to\infty} \{\hat{\rho}_i^{d(i,k)} \} \leq \bar{\rho}.
$$
Since the choice of $i\in\Iforw$ was arbitrary, we know that
$\{\hat\rho_i^{d(i,k)}\}_{k\in\nN}$ is bounded for all $i\in\Iforw$, and the first
phase of the proof is complete.

We now turn to lemmas~\ref{lemIsProject}-\ref{LemPsiAndPhi}. First, Lemma
\ref{lemIsProject} still holds, since it remains true that $y_i^k = T_i x_i^k$
for all $i\in\Iforw$ and $k\geq M$.  Next, a result like that of
Lemma~\ref{lemUpdates} holds, but with  $\rho_i^{l(i,k)}$ replaced by  $\hat{\rho}_i^{l(i,k)}$ for all $i\in\Iforw$. The
arguments of Lemmas \ref{LemBoundOnl}-\ref{LemPsiAndPhi} remain completely
unchanged.
 
Next we show that Lemma \ref{lemBoundedGrad} holds with a different
constant.  The derivation leading up to (\ref{eqEndBounded}) continues to
apply if we incorporate the substitution in Lemma~\ref{lemUpdates} specified
in the previous paragraph. Therefore, we replace 
$\rho_i^k$ by $\hat{\rho}_i^k$ in~\eqref{eqEndBounded} for $i\in\Iforw$.  Using
\eqref{eqRhoHatLB}/\eqref{eqRhoHatLB2} and the fact that $\limsup_{k\to\infty} \{\hat{\rho}_i^{d(i,k)} \} \leq
\bar{\rho}$ we conclude that 
Lemma~ \ref{lemBoundedGrad} still holds, with the constant $\xi_1$ adjusted in
light of~\eqref{eqRhoHatLB}/\eqref{eqRhoHatLB2}. 

Now we show that Lemma~\ref{lemLowerBoundPhi} holds with a different
constant. For  $k\geq M$, we may use Lemma~\ref{lemUpdates} and the
termination criterion for Algorithm~\ref{AlgBackTrack} to write
\begin{eqnarray}
\psi_k
&=&
\sum_{i\in\Iback}
\left\langle 
G_i z^{l(i,k)} - x_i^k
,
y_i^k - w_i^{l(i,k)}
\right\rangle 
+
\sum_{i\in\Iforw}
\left\langle 
G_i z^{l(i,k)} - x_i^k
,
y_i^k - w_i^{l(i,k)}
\right\rangle 
\nonumber\\
&\geq&
(1-\sigma)\sum_{i\in\Iback}(\rho_i^k)^{-1}\|x_i^k - G_{i}z^{l(i,k)}\|^2
+
\Delta
\sum_{i\in\Iforw}
\|x_i^k - G_{i}z^{l(i,k)}\|^2.
\nonumber
\end{eqnarray}
Here, the terms involving $\Iback$ are dealt with the same way as before in
Lemma \ref{lemLowerBoundPhi}.
We conclude that Lemma~\ref{lemLowerBoundPhi} holds with $\xi_2$ replaced by 
$$
\xi_2' = 
\min
\left\{
(1-\sigma)\overline{\rho}^{-1},\Delta
\right\}.
$$

Now we show that Lemma~\ref{lemLBPhi2} holds with a different constant.
The derivation up to~\eqref{eqPassLemFin} proceeds as before, but
replacing $\rho_i^{l(i,k)}$ with $\hat{\rho}_i^{l(i,k)}$ for $i\in\Iforw$. 
Using~\eqref{eqRhoHatLB}--\eqref{eqRhoHatLB2} and Assumption~\ref{assErr}, it
is clear that the conclusion of Lemma~\ref{lemLBPhi2} follows with the
constant $\xi_3$ adjusted in light of \eqref{eqRhoHatLB}--\eqref{eqRhoHatLB2}.

Finally, the rest of the proof now follows in the same way as in the proof of
Theorem~\ref{ThmAsync}.
\ourqed 
\end{proof}


\subsection{Backtracking is Unnecessary for Affine Operators}\label{secOptStep}
When $i\in\Iforw$ and $T_i$ affine, it is not necessary to iteratively
backtrack to find a valid stepsize.  Instead, it is possible to directly solve
for a stepsize $\rho = \rho_i^{(j,k)}$ such that the condition on
line~\ref{lineIf} of Algorithm~\ref{AlgLineSearch} is immediately satisfied.
Thus, one can process an affine operator with only two forward steps, even
without having estimated its Lipschitz constant.

From here on, we continue to use the notation $\theta_i^k = G_i z^{d(i,k)}$
and $\zeta_i^k = T_i\theta_i^k$ introduced in Algorithm~\ref{AlgBackTrack}.
Fix $i\in\Iforw$ and suppose that $T_i x = T_i^lx+c_i$ where $c_i\in \calH_i$
and $T_i^l$ is linear. 
The loop termination condition on line~\ref{lineIf} of
Algorithm~\ref{AlgBackTrack} may be written
\begin{align}\label{eqIfCond}
\langle
\theta_i^k-\tilde{x}_i^{(j,k)},\tilde{y}_i^{(j,k)}-w_i^{d(i,k)}\rangle
\geq
\Delta\|\theta_i^k-\tilde{x}_i^{(j,k)}\|^2.
\end{align}
Substituting the expressions for 
$\tilde{x}_i^{(j,k)}$ and $\tilde{y}_i^{(j,k)}$ from
lines~\ref{lineX}-\ref{lineY} of Algorithm~\ref{AlgBackTrack} into the
left-hand side of~\eqref{eqIfCond}, replacing $\rho_i^{(i,j)}$ with $\rho$ for
simplicity, and using the linearity of $T_i^l$ yields
\begin{eqnarray}
&&
\rho 
\left\langle 
\zeta_i^k - w_i^{d(i,k)} 
,
T_i^l \!
\left(
\theta_i^k - \rho\big(T_i G_i z^{d(i,k)} - w_i^{d(i,k)}\big)
\right)
+c_i
-w_i^{d(i,k)}
\right\rangle 
\nonumber\\
&=&
\rho
\left\langle 
\zeta_i^k - w_i^{d(i,k)} 
,
T_i^l
\theta_i^k - \rho T_i^l\!\big(\zeta_i^k - w_i^{d(i,k)}\big)
+c_i
-w_i^{d(i,k)}
\right\rangle
\nonumber\\
&=&
\rho
\left\langle 
\zeta_i^k - w_i^{d(i,k)} 
,
\zeta_i^k 
-w_i^{d(i,k)}
- \rho T_i^l\!\big(\zeta_i^k - w_i^{d(i,k)}\big)
\right\rangle
\nonumber\\
&=&
\rho
\left(
\|\zeta_i^k - w_i^{d(i,k)} \|^2
-
\rho
\left\langle
\zeta_i^k - w_i^{d(i,k)} 
,
T_i^l\!\big(\zeta_i^k - w_i^{d(i,k)}\big)
\right\rangle 
\right).\label{eqLHS}
\end{eqnarray}
Substituting the expression for $\tilde{x}_i^{(i,j)}$ from line~\ref{lineX} of
Algorithm~\ref{AlgBackTrack}, the right-hand side of \eqref{eqIfCond} may be
written
\begin{align}\label{eqRHS}
\Delta\rho^2\|\zeta_i^k - w_i^{d(i,k)}\|^2.
\end{align}
Substituting~\eqref{eqLHS} and~\eqref{eqRHS} into~\eqref{eqIfCond} and solving
for $\rho$ yields that the loop exit condition holds when
\begin{equation}
\rho \leq \tilde{\rho}_i^k \triangleq
\frac{\|\zeta_i^k - w_i^{d(i,k)}\|^2}
{
\Delta\|\zeta_i^k - w_i^{d(i,k)}\|^2
+
\left\langle
\zeta_i^k - w_i^{d(i,k)} 
,
T_i^l\!\big(\zeta_i^k - w_i^{d(i,k)}\big)
\right\rangle 	
}.\label{eqAutoStep}
\end{equation}
If $\zeta_i^k - w_i^{d(i,k)} = 0$, then \eqref{eqAutoStep} is not defined,
but in this case the step acceptance condition~\eqref{eqIfCond} holds
trivially and lines~\ref{lineX}-\ref{lineY} of the backtracking procedure
yield $\tilde{x}_i^{(j,k)} = \theta_i^k$ and $\tilde y_i^{(j,k)} = \zeta_i^k$
for any stepsize $\rho_i^{(j,k)}$.

We next show that $\tilde{\rho}_i^k$ as defined in~\eqref{eqAutoStep} will
behave in a bounded manner even as $\zeta_i^k - w_i^{d(i,k)}\to 0$.
Temporarily letting $\xi = \zeta_i^k - w_i^{d(i,k)}$, we note that as long
as $\xi \neq 0$, we have
\begin{equation}
\tilde{\rho}_i^k
=
\frac{\norm{\xi}^2}{\Delta\norm{\xi}^2 + \inner{\xi}{T_i^l \xi}}
=
\frac{1}{\Delta + 
	\frac{\langle \xi,T_i^l\xi\rangle}{\|\xi\|^2}
}
\in
\left[
\frac{1}{\Delta+L_i}
,
\frac{1}{\Delta}
\right],\label{eqAutoBounds}
\end{equation}
where the inclusion follows because $T_i$ is monotone and thus $T_i^l$ is
positive semidefinite, and because $T_i$ is $L_i$-Lipschitz continuous and
therefore so is $T_i^l$.
Thus, choosing $\tilde{\rho}_i^k$ to take some arbitrary 
fixed value $\bar{\rho} > 0$
whenever $\zeta_i^k - w_i^{d(i,k)}=0$, the sequence $\{\tilde{\rho}_i^k\}$
is bounded from both above and below, and all of the arguments of
Theorem~\ref{thmBackTrack} apply if we use $\tilde{\rho}_i^k$ in place of the
results of the backtracking line search.

To calculate~\eqref{eqAutoStep}, one must compute $\zeta_i^k =
T_iG_i z^{d(i,k)}$ and $T_i^l\big(\zeta_i^k - w_i^{d(i,k)}\big)$. Then
$x_i^k$ can be obtained via $x_i^{k} = \theta_i^k - \rho\big(\zeta_i^k
- w_i^{d(i,k)}\big)$ and 
\begin{align}
y_i^k =
\zeta_i^k 
- \rho T_i^l\!\big(\zeta_i^k - w_i^{d(i,k)}\big).
\label{yNew}
 \end{align}
In total, this procedure requires one application of $G_i$ and two
of $T_i^l$.




\subsection{Greedy Block Selection}\label{secGreed}
We now introduce a greedy block selection strategy which may be useful in some
\col{block-iterative} implementations of Algorithm~\ref{AlgfullyAsync}, 
\col{such as Algorithm \ref{AlgBlockIter}. 
In essence, this selection strategy provides a way to pick $I_k$ at each
iteration in Algorithm \ref{AlgBlockIter}, and we have found it to improve
performance on several empirical tests.}

Consider Algorithm~\ref{AlgBlockIter} with 
$\lvert I_k \rvert = 1$ for all $k$ (only one
subproblem activated pe{}r iteration), and $\beta_k = 1$ for all $k$ (no
overrelaxation of the projection step).  Consider some particular iteration $k
\geq M$ and assume $\|\nabla\varphi_k\|>0$ (otherwise the algorithm terminates at 
a solution).  
Ideally, 
one might like to maximize the length of
the step $p^{k+1} - p^k$ toward the solution set $\calS$, and
$\norm{p^{k+1}-p^k}_\gamma = \varphi_k(p^k)/\norm{\nabla
\varphi_k}_\gamma$.



Assuming that $\beta_k=1$,
 the current point $p^{k}$ computed on
 lines~\ref{eqAlgproj1}-\ref{eqAlgproj2} of Algorithm~\ref{AlgfullyAsync} is
 the projection of $p^{k-1}$ onto the halfspace $\{p:\varphi_{k-1}(p)\leq
 0\}$.  If $p^{k-1}$ was not already in this halfspace, that is,
 $\varphi_{k-1}(p^{k-1}) > 0$, then after the projection we have
 $\varphi_{k-1}(p^k) = 0$.
 Using the notation $G_n = I$ and $w_n^k$ defined in~\eqref{defwn},
 $\varphi_{k-1}(p^k) = 0$ is equivalent to
\begin{equation}
\sum_{i=1}^{n}
\langle
G_i z^{k} - x_i^{k-1}
,
y_i^{k-1} - w_i^{k}
\rangle
= 
0.\label{eqGreedySum}
\end{equation}


Suppose we select operator $i$ to be processed next, that is, $I_k=\{i\}$.
After updating $(x_i^k,y_i^k)$, the corresponding term in the summation
in~\eqref{eqGreedySum} becomes \col{bounded below by $\xi\|G_i z^k - x_i^k\|^2
\geq 0$, where $\xi = (1-\sigma)/\rho_i^k$ for $i\in \Iback$,
$\xi = \Delta$ for  $i\in \Iforw$ with backtracking, and $\xi =
 \bar{\rho}_i^{-1} - L_i$ for $i\in \Iforw$ without backtracking. In any
 case, processing operator $i$ will cause the $i^{\text{th}}$ term} to become
 nonnegative while the other terms remain unchanged, so if we select an $i$
 with $\langle G_i z^{k} - x_i^{k-1}, y_i^{k-1} - w_i^{k}\rangle < 0$, then
 the sum in~\eqref{eqGreedySum} must increase by at least $-
\langle G_i z^{k} - x_i^{k-1}, y_i^{k-1} - w_i^{k}\rangle$, meaning that after
processing subproblem $i$ we will have
\[
\varphi_k(p^k) \geq -
\langle G_i z^{k} - x_i^{k}, y_i^{k} - w_i^{k}\rangle > 0.
\]
Choosing the $i$ for which $\langle G_i z^{k} - x_i^{k-1}, y_i^{k-1} -
w_i^{k}\rangle$ is the most negative maximizes the above lower bound on
$\varphi_k(p^k)$ and would thus seem a promising heuristic for selecting $i$.

Note that this ``greedy'' procedure is only heuristic because it does not take
into account the denominator in the projection operation, nor how much
$\langle G_i z^{k} - x_i^{k}, y_i^{k} - w_i^{k}\rangle$ might exceed zero after processing block $i$.  Predicting this quantity for every block, however, might require 
essentially the same computation as evaluating a proximal or forward step for
 all blocks, after which we might as well update all blocks, that is, set $I_k
 = \{1,\ldots,n\}$. 

 In order to guarantee convergence under this block selection heuristic, we
 must include some safeguard to make sure that
 Assumption~\ref{assAsync}(\ref{quasicyclic}) holds.  One straightforward
 option is as follows: if a block has not been processed for more than $M>0$
 iterations, we must process it immediately regardless of the value of
 $\langle G_i z^{k} - x_i^{k-1}, y_i^{k-1} - w_i^{k}\rangle$. 

\subsection{Variable Metrics}
Looking at Lemmas \ref{PositivePhi} and \ref{lemLBPhi2}, it can be seen that
the update rules for $(x_i^k,y_i^k)$ can be abstracted. In fact any procedure
that returns a pair $(x_i^k,y_i^k)$ in the graph of $T_i$ satisfying, for some $\xi_4>0$,
\begin{align}\label{eqVMineq1}
&(\forall i=1,\ldots,n)\quad
\langle G_i z^{l(i,k)} - x_i^k,y_i^k - w_i^{l(i,k)}\rangle
\geq
\xi_4\|G_i z^{l(i,k)} - x_i^k\|^2
\\\label{eqVMineq2}
&(\forall i\in\Iback)\quad
\langle G_i z^{l(i,k)} - x_i^k,y_i^k - w_i^{l(i,k)}\rangle
\geq
\xi_4\|y_i^k - w_i^{l(i,k)}\|^2
\\
\nonumber 
&(\forall i\in\Iforw)\quad
\langle G_i z^{l(i,k)} - x_i^k,y_i^k - w_i^{l(i,k)}\rangle
+
L_i\|G_i z^{l(i,k)} - x_i^k\|^2
\\\label{eqVMineq3}
&
\qquad\qquad\qquad\qquad\quad\qquad\qquad\qquad\qquad
\geq\xi_4\|T_i G_i z^{l(i,k)} - w_i^{l(i,k)}\|^2
\end{align}
yields a convergent algorithm. As with lemmas \ref{PositivePhi} and
\ref{lemLBPhi2}, these inequalities need only hold in the limit.

An obvious way to make use of this abstraction is to introduce variable metrics. To simplify the following, we will ignore the error terms $e_i^k$ and assume no delays, i.e. $d(i,k)=k$. The updates on lines \ref{lineaupdate}--\ref{lineBackwardUpdateY} and \ref{LineForwardUpdate}--\ref{ForwardyUpdate} of Algorithm \ref{AlgfullyAsync} can be replaced with
\begin{align}
&(\forall i\in\Iback)&\quad
x_i^k + \rho_i^{k} U_i^k y_i^k 
&= G_i z^{k}+\rho_i^{k}U_i^k w_i^{k},\quad 
&y_i^k\in T_i x_i^k&,
\label{eqVMprox}\\\label{eqVMforw}
&(\forall i\in\Iforw)&\quad
x_i^k 
&= z^{k} - \rho_i^{k}U_i^k(T_i G_i z^{k} - w_i^{k}),\quad 
&y_i^k = T_i x_i^k&,
\end{align}
where $\{U_i^k:\calH_i\to\calH_i\}$ are a sequence of bounded linear self-adjoint operators such that 
\begin{align}
\forall i=1,\ldots,n, x\in \calH_i:
\quad  
\inf_{k\geq 1}\langle x,U_i^k x\rangle \geq \ulambda \|x\|^2
\text{ and }
\sup_{k\geq 1} \|U_i^k\|\leq \olambda \label{eqEig}
\end{align}
where $0<\ulambda,\olambda<\infty$. In the finite dimensional case, 
\eqref{eqEig} simply states that the eigenvalues of the set of matrices $\{U_i^k\}$ can be uniformly bounded away from $0$ and $+\infty$. 
It can be shown that using  \eqref{eqVMprox}--\eqref{eqVMforw} leads to the desired inequalities \eqref{eqVMineq1}--\eqref{eqVMineq3}. 

The new update \eqref{eqVMprox} can be written as
\begin{align}\label{eqNewProx}
x_i^k = (I+\rho_i^{k} U_i^k T_i)^{-1}
(
G_i z^{k}+\rho_i^{k}U_i^k w_i^{k}
).
\end{align} 
It was shown in~\cite[Lemma 3.7]{combettes2014variable} that this is a
proximal step with respect to $U_i^k T_i$ and that this operator is maximal
monotone under an appropriate inner product. Thus the update \eqref{eqNewProx}
is single valued with full domain and hence well-defined. In the optimization
context where $T_i = \partial f_i$ for closed convex proper $f_i$, solving
\eqref{eqNewProx} corresponds to the subproblem
\begin{align*}
\min_{x\in\calH_i} 
\left\{
\rho_i^k f_i(x)
+
\frac{1}{2}
\langle
(U_i^k)^{-1}(x-a)
,
x-a
\rangle\right\} 
\end{align*}
where 
$a = G_i z^{k}+\rho_i^{k}U_i^k w_i^{k}$.
For the variable-metric forward step~\eqref{eqVMforw}, the stepsize constraint
\eqref{eqForwUpper} must be replaced by $\rho_i^k < {1}/{\|U_i^k\|L_i}$.
 
\newcommand{\calT}{\mathcal{T}}
\newcommand{\bbeta}{\boldsymbol{\beta}}
\newcommand{\bgamma}{\boldsymbol{\gamma}}
\newcommand{\cp}{\texttt{cp-bt}}
\newcommand{\tseng}{\texttt{tseng-pd}}
\newcommand{\frb}{\texttt{frb-pd}}
\newcommand{\psb}{\texttt{psb}}
\newcommand{\psf}{\texttt{psf}}
\newcommand{\psfg}{\texttt{psf-g}}
\newcommand{ \psbg}{\texttt{psb-g}}

\section{Numerical Experiments}
\label{secNumerical}
We now present some preliminary numerical experiments with 
Algorithm~\ref{AlgBlockIter},
evaluating various strategies for selecting $I_k$
and comparing efficiency of forward and (approximate) backward steps. 
All our numerical experiments were implemented in Python (using
\texttt{numpy} and \texttt{scipy}) on an Intel Xeon workstation running Linux. 

\subsection{Rare Feature Selection}
The work in~\cite{2018arXiv180306675Y} studies the problem of 
utilizing rare features in
machine learning problems.  In this context, a ``rare feature'' is one whose value is 
rarely nonzero, making it hard to estimate the corresponding model
coefficients accurately. Despite this, such features can be highly
informative, so the standard practice of discarding them is wasteful. The
technique in \cite{2018arXiv180306675Y} overcomes this difficulty by making
use of an auxiliary tree data structure $\calT$ describing feature relatedness. 
Each leaf of the tree is a feature and two features' closeness on
the tree measures how ``related" they are. Closely related features can then
be aggregated (summed) so that more samples are captured, increasing the
accuracy of the coefficient estimate for a single coefficient for the aggregated 
features.

To formulate the resulting aggregation and fitting problem,
\cite{2018arXiv180306675Y} introduced the following generalized regression
problem:
\begin{align}\label{eqTheirTrip}
\min_{\substack{\bbeta\in \rR^d, \bgamma\in \rR^{|\calT|}}}
\Big\{ \left. \!
\ell(X\bbeta,b)
+
\lambda 
\big(
(1-\alpha)\|\bbeta\|_1
+
\alpha\|\bgamma_{-r}\|_1
\big)
~ \right| \,\, \bbeta = H\bgamma 
\Big\}
\end{align}
where $\ell:\rR^m\times \rR^m \to \rR$ is a loss function, $X\in
\rR^{m\times d}$ is the data matrix, $b\in\rR^m$ is the target (response)
vector, and $\bbeta\in\rR^d$ are the feature coefficients. Each $\bgamma_i$ is
associated with a node of the similarity tree $\calT$, and $\bgamma_{-r}$
denotes the subvector of $\bgamma$ corresponding to all nodes except the root
node. The matrix $H\in\rR^{d\times |\calT|}$ contains a $1$ in positions $i,j$
for those features $i$ which correspond to a leaf of $\calT$ that is descended
from node $j$, and elsewhere contains zeroes. Due to the constraint $\bbeta =
H\bgamma$, the coefficient $\bgamma_j$ of each tree node $j$ contributes
additively to the coefficient $\bbeta_i$ of each feature descended from $j$.
$H$ thus fuses coefficients together in the following way: if $\bgamma_i$ is
nonzero for a node $i$ and all descendants of $\bgamma_i$ in $\calT$ are $0$,
then all coefficients on the leaves which are descendant from $\bgamma_i$ are
equal (see \cite[Sec. 3.2]{2018arXiv180306675Y} for more details). The
$\ell_1$ norm on $\bgamma$ enforces sparsity of $\bgamma$, which in turn fuses
together coefficients in $\bbeta$ associated with similar features. The
$\ell_1$ norm on $\bbeta$ itself additionally enforces sparsity on these
coefficients, which is also desirable. The model can allow for an offset
variable by incorporating columns/rows of $1$'s and $0$'s in $X$ and $H$, but
for simplicity we omit the details.

\subsection{TripAdvisor Reviews}

We apply this model to a dataset of TripAdvisor reviews of hotels from
\cite{2018arXiv180306675Y}. The response variable was the overall review of
the hotel in the set $\{1,2,3,4,5\}$. 
The features were the counts of certain
adjectives in the review. Many adjectives were very rare, with $95\%$ of the
adjectives appearing in fewer than $5\%$ of the reviews. The authors
of~\cite{2018arXiv180306675Y} constructed a similarity tree using information
from word embeddings and emotion lexicon labels; there are
$7,\!573$ adjectives from the $209,\!987$ reviews and the tree $\calT$ had
$15,\!145$ nodes. A test set of $40,\!000$ examples was withheld, leaving a sparse
$169,\!987\times 7,\! 573$ \
matrix $X$ having only $0.32\%$
nonzero entries. The $7,\! 573\times 15,\!145 $ matrix  $H$ arising from the
similarity tree $\calT$ is also sparse, having $0.15\%$ nonzero entries. In
our implementation, we used the sparse matrix package \texttt{sparse} in
\texttt{scipy}. 

In \cite{2018arXiv180306675Y}, the elements of $b$ are the review ratings and
the loss function is given by the standard least-squares formula
$\ell(X\bbeta,b)=(1/2m)\|X\bbeta-b\|_2^2.$ To emphasize the advantages of our
new forward-step version of projective splitting over previous backward-step
versions, we instead use the same data and regularizers to construct a
classification problem with the logistic loss.  We assigned the $73,\! 987$
reviews with a rating of $5$ a value of $b_i=+1$, while we labeled the
$96,\!000$ reviews with value $4$ or less with $b_i=-1$. The loss is then
\begin{align}\label{eqLR}
\ell(X\bbeta,b) = 
  \frac{1}{m}\sum_{j=1}^m 
     \log\!\Big(1+\exp\!\big(\!-b_j \langle x_j,\bbeta\rangle\big)\Big)
\end{align}
where $x_j$ is the $j$th row of $X$.  The classification problem is then to
predict which reviews are associated with a rating of $5$.  

In practice, one typically would solve \eqref{eqTheirTrip} for many values of
$(\alpha,\lambda)$ and then choose the final model based on cross validation.
To assess the computational performance of the tested methods, we solve three
representative examples corresponding to sparse, medium, and dense solutions. 
The corresponding values for $\lambda$ were
\col{$\{10^{-8},10^{-6},10^{-4}\}$}.  In preliminary experiements, we found that the
value of $\alpha$ had little effect on algorithm performance, so we fixed
$\alpha=0.5$ for simplicity.

\subsection{Applying Projective Splitting}

The work in \cite{2018arXiv180306675Y} solves the problem \eqref{eqTheirTrip},
with $\ell$ set to the least-squares loss, using a specialized application
of the ADMM. The implementation involves precomputing the singular value
decompositions (SVDs) of the (large) matrices $X$ and $H$, and so does not fall
within the scope of standard first-order methods. Instead, we solve
\eqref{eqTheirTrip} with the logistic loss by simply eliminating $\bbeta$, so
that the formulation becomes
\begin{align}\label{eqOurTrip}
F^*\triangleq \min_{\substack{\bgamma\in \rR^{|\calT|}}} \!\!
\Big\{
\ell(X H \bgamma,b)
+
\lambda 
\left(
(1-\alpha)\|H\bgamma\|_1
+
\alpha\|\bgamma_{-r}\|_1
\right)
\Big\}.
\end{align}
To utilize  block-iterative updates in Algorithm \ref{AlgBlockIter}, we split
 up the loss function as follows: Let $\calR = \{R_1,..,R_P\}$ be an arbitrary
 partition of $\{1,\ldots,m\}$. For $i=1,\ldots,P$, let
 $X_i\in\mathbb{R}^{|R_i|\times d}$ be the submatrix of $X$ with rows
 corresponding to indices in $R_i$ and similarly let
 $b^i\in\mathbb{R}^{|R_i|}$ be the corresponding subvector of $b$.  Then
 \eqref{eqOurTrip} is equivalent to
\begin{align}\label{eqOurTripGreedy}
\min_{\substack{\bgamma\in \rR^{|\calT|}}} \!\!
\left\{
\sum_{i=1}^P
\ell(X_i H \bgamma,b^i)
+
\lambda 
\left(
(1-\alpha)\|H\bgamma\|_1
+
\alpha\|\bgamma_{-r}\|_1
\right)
\right\}.
\end{align}
There are several ways to formulate this problem as a special case of
\eqref{ProbOpt}, leading to different realizations of Algorithm
\ref{AlgBlockIter}. The approach that we found to give the best empirical
performance was to set \col{$n=P+3$ and
\begin{align*}
G_i &= H & f_i(t) &= \ell(X_it,b^i)  &i &=1,\ldots n-3 \\
G_{n-2} &= H & f_{n-2}(t)&=\lambda(1-\alpha)\|t\|_1 \\
G_{n-1} &=\tilde{G}
& f_{n-1}(t)&=\lambda\alpha\|t\|_1 \\
G_n &= I& f_n(t)&=0,
\end{align*}
where
\begin{align*}
\tilde{G} : [\bgamma_1 \;\; \bgamma_2 \;\; \cdots \;\; \bgamma_{|\calT|-1} \;\; \bgamma_{|\calT|}]
\mapsto
[\bgamma_1 \;\; \bgamma_2 \;\; \cdots \;\; \bgamma_{|\calT|-1}],
\end{align*}
and the last element of $\bgamma$, $\bgamma_{|\calT|}$, is the root of the tree.  
We append the trivial function $f_n = 0$ in order to comply with the requirement 
that the final linear operator $G_n$ be the identity; see 
\eqref{probCompMonoMany}. }
The functions $f_{n-2}$ and $f_{n-1}$ have easily-computed proximal operators,
so we process them at every iteration. \col{Further, the proximal operator
of $f_n$ has is simply the identity, so we also process it at each iteration. }
Therefore, $\{n-2,n-1,n\}\subseteq I_k$ for
all $k\geq 1$. On the other hand, the functions $f_i(t)$ for $i=1,\ldots,P$ are
$$
f_i(t) = \ell(X_i t,b) = \frac{1}{m}\sum_{j=1}^{|R_i|} 
   \log\!\Big(1+\exp\!\big(\!-\!b^i_j \langle x_{ij},t\rangle\big)\Big),
$$
where $x_{ij}$ is the $j$th row of the submatrix $X_i$ and $b_j^i$ is the
$j$th element of $b^i$. These functions are Lipschitz differentiable and so
may be processed by our new forward-step procedure. We use the backtracking
procedure in Section \ref{secLineSearch} so that we do not need to estimate
the Lipschitz constant of each $\ell_i$, a potentially costly computation
involving the SVD of each $X_i$. The most time-consuming part of each gradient
evaluation are two matrix multiplications, one by $X_i$ and one by $X_i^\top$.
We will refer to the approach of setting $\Iforw = \{1,\ldots,P\}$ and using
backtracking as ``Projective Splitting with Forward Steps'' (\psf).

On the other hand, even though the proximal operators of $f_1,\ldots,f_{P}$
lack a closed form, it is still possible to process these functions with an
approximate backward step. 
The exact proximal map for $f_i$
is the solution to 
\begin{eqnarray}\label{eqProxLR}
\argmin_{t\in\rR^d} 
\left\{
\frac{1}{m}\sum_{j=1}^{|R_i|} 
   \log\!\Big(1+\exp\!\big(\!-\!b^i_j \langle x_{ij},t\rangle\big)\Big)
+
\frac{1}{2}\|t - u\|_2^2
\right\}.
\end{eqnarray}
This is an unconstrained nonlinear convex program and there are many different
ways one could approximately solve it. Since we are interested in scalable
first-order approaches, we chose the L-BFGS method --- see for
example~\cite{LBFGS80} --- which has small memory footprint and only requires
gradient and function evaluations. So, we choose some $\sigma\in[0,1)$ and
apply L-BFGS to solve \eqref{eqProxLR} until the relative error
criteria~\eqref{err1} and~\eqref{err2} are met.

For a given candidate solution $x_{i}^k$, we have $y_{i}^k = \nabla \ell (X_i
x_i^k,b_i)$, and the error can be explicitly computed as $e_{i}^k = x^k_{i}
+\rho_{i} y^k_{i} -(H z^{k}+\rho_{i} w_{i}^{k})$.  Every iteration of L-BFGS
requires at least one gradient and function evaluation, which in turn requires
two matrix multiplies, one by $X_i$ and one by $X_i^\top$. We ``warm-start"
 L-BFGS by initializing it at $x^{k-1}_i$. We will refer to this approach
as ``Projective Splitting with Backward Steps'' (\psb).

The coordination procedure (lines \ref{lineCoordStart}--\ref{lineCoordEnd}) is
the same for \psf\space and \psb, requiring two multiplies by $H$, two
by $H^\top$, vector additions, inner products, and scalar multiplications.

We tried $P=1$ and $P=10$, with each block chosen to have the same number of elements (to within $P$,
since $m$ is not divisible by $P$) of contiguous rows from $X$.
At each iteration, we
selected one block from among $1,\ldots,P$ for a forward step in \psf\space or
backward step with L-BFGS in \psb, and blocks $P+1$, $P+2$, and $P+3$ for backward steps.
Thus, $I_k$ always has the form $\{i,P+1,P+2,P+3\}$, with $1 \leq i\leq P$. To
select this $i$, we tested three strategies: the greedy block selection scheme described in
Section~\ref{secGreed}, choosing blocks at random, \col{and cycling through the blocks in a round-robin fashion}. 
For the greedy
scheme, we did not use the safeguard parameter $M$ as in practice we found
that every block was updated fairly regularly. 

We refer to the greedy variants with $P=10$ blocks as \texttt{psf-g} and \texttt{psb-g}, those with randomly selected blocks as \texttt{psf-r} and \texttt{psb-r}, \col{and those with cyclically selected blocks as \texttt{psf-c} and \texttt{psb-c}}. Finally, the versions with $P=1$ are referred to as \texttt{psf-1} and \texttt{psb-1}. 


\subsection{The Competition}
To compare with our proposed methods, we restricted our attention to
algorithms with comparable features and benefits. In particular, we only
considered first-order methods which do not requre computing Lipschitz
constants of gradients and matrices. Very
few such methods apply to \eqref{eqOurTrip}. The presence of the matrix $H$ in
the term $\|H\bgamma\|_1$ makes it difficult to apply Davis-Yin three-operator
splitting~\cite{davis2015three} 
and related methods~\cite{pedregosa2018adaptive}, since the proximal
operator of this function cannot be computed in a simple way.  
We compared our projective splitting methods with
the following methods:
\newcommand{\proxsg}{\texttt{prox-sgd}}
\begin{itemize}
\item The backtracking linesearch variant of the Chambolle-Pock primal-dual
splitting method \cite{malitsky2018first}, which we refer to as \cp.
\item The algorithm of \cite{combettes2012primal}. This approach is based 
on the
``monotone + skew" inclusion formulation obtained by first
defining the monotone operators 
\begin{align*}
T_1(\bbeta) &= \lambda(1-\alpha)\partial \|\bbeta\|_1 & 
T_2(\bgamma) &=\lambda\alpha\partial\|\bgamma_{-r}\|_1 &
T_3(\bgamma) &= \nabla_{\bgamma}\big[\ell (XH\bgamma,b)\big],
\end{align*}
and then formulating
the problem as $0\in \tilde{A}(z,w_1,w_2) + \tilde{B}(z,w_1,w_2)$, 
where $\tilde{A}$ and $\tilde{B}$ are defined by
\begin{align}\label{eqMonPlSk1}
\tilde{A}(z,w_1,w_2) &= \{0\} \times T_1^{-1}w_1 \times T_{2}^{-1}w_{2}
\\ \label{eqMonPlSk2}
\tilde{B}(z,w_1,w_2) &=
\left[
\begin{array}{c}
T_3(z)\\
0\\
0
\end{array}
\right]
+
\left[
\begin{array}{ccc}
0&H^{\top}&I\\
-H&0& 0\\
-I & 0 & 0
\end{array}
\right]
\left[
\begin{array}{c}
z\\
w_1\\
w_{2}
\end{array}
\right].
\end{align}
$\tilde{A}$ is maximal monotone, while $\tilde{B}$ is the sum of two
Lipshitz monotone operators (the second being skew linear), and therefore is also
Lipschitz monotone. The algorithm in \cite{combettes2012primal} is
essentially Tseng's forward-backward-forward method~\cite{tseng2000modified}
applied to this inclusion, using resolvent steps for $\tilde{A}$ and forward
steps for $\tilde{B}$. Thus, we call this method \tseng. In order to achieve
good performance with \tseng\space we had to incorporate a diagonal
preconditioner as proposed in \cite{vu2013variable}. We used the following preconditioner: 
\begin{align}\label{eqPreCond}
U = \text{diag}(I_{d\times d},\gamma_{pd} I_{d\times d},\gamma_{pd} I_{d\times d})
\end{align}
where $U$ is used as in \cite[Eq.~(3.2)]{vu2013variable} for \tseng.
\item The recently proposed forward-reflected-backward
method~\cite{tam2018forward}, applied to this same primal-dual inclusion $0\in
\tilde{A}(z,w_1,w_2) + \tilde{B}(z,w_1,w_2)$ specified by
\eqref{eqMonPlSk1}-\eqref{eqMonPlSk2}.  We call this method \frb. For
this method, we used the same preconditioner given in \eqref{eqPreCond},
used as $M^{-1}$ on \cite[p.~7]{tam2018forward}.
\end{itemize}

 \subsection{Algorithm Parameter Selection}

For \psf, we used the backtracking procedure of Section \ref{secLineSearch} with
$\Delta=1$ to determine $\rho_1^k,\ldots,\rho_{n-3}^k$.  \col{For the stepsizes associated with the regularizers, we simply set $\rho_{n-2}^k=\rho_{n-1}^k=\rho_n^k=1$. }  For backtracking in all methods, we set the trial
stepsize equal to the previously discovered stepsize. 

\col{For \psb, we used 
$\rho_1^k=\ldots=\rho_{n}^k=1$ for simplicity.}
For the L-BFGS procedure in 
\psb, we set the history parameter to be $10$ (\emph{i.e.} 
the past $10$ variables and
gradients were used to approximate the Hessian). We used a Wolfe linesearch
with $C_1=10^{-4}$ and $C_2=0.9$.

Each tested method then had one additional tuning parameter: $\beta$ given in line 2.a of Algorithm 4
of~\cite{malitsky2018first} for \cp, $\gamma_{pd}$ given in \eqref{eqPreCond} for
\tseng\space and \frb, and $\gamma$
for \psf\space and \psb. The values we used are given in Table
\ref{table-tune}. These values were chosen by running each method for $2000$ 
iterations 
and picking the tuning parameter from $\{10^{-6},10^{-5},\ldots,10^5,10^6\}$
giving the smallest final function value. We then ran a longer
experiment (about 10 minutes) for each method, using the chosen tuning
parameter. The greedy, random, \col{cyclic,} and $1$-block 
variants of \psf\space and
\psb\space all used the same tuning parameter values.
\col{ 
\begin{table}[h]
	\centering
	\begin{tabular}{c|c|c|c|c}
		parameter & method & $\lambda=10^{-8}$ & $\lambda=10^{-6}$ & $\lambda=10^{-4}$\\
		\hline\hline 
		\vspace{-0.2cm}
		&&&&\\
		$\gamma$  & \psf & $10^{-5} $ & $10^{-6}$ & $10^{-4}$ \\	
		$\gamma$  & \psb & $10^{-4} $ & $10^{-4}$ & $10^{-4}$ \\	
		$\beta$  & \cp & $10^6$ & $10^6$ & $10$\\
		$\gamma_{pd}$ & \tseng& $1$ & $1$ & $10$ \\
		$\gamma_{pd}$ & \frb& $1$ & $1$ & $100$ \\
		\hline 
	\end{tabular}
	\caption{Tuning parameters for the \eqref{eqOurTrip} applied to TripAdvisor data.}
	\label{table-tune}
\end{table}
}
\subsection{Results}
In Figure \ref{Fig-results} we plot the \col{objective function values}
against elapsed wall-clock running time, excluding time to compute the plotted function values.
\col{For \psf\space and \psb, we computed function values for the primal variable $z^k$.}
For \cp, we computed the objective at $y^{k}$ as
given in \cite[Algorithm 4]{malitsky2018first}. For \tseng\space and
\frb, we computed the objective values for the primal iterate corresponding to $z$ in
\eqref{eqMonPlSk1}-\eqref{eqMonPlSk2}.

The best performing variants of projective splitting were \texttt{psf-g} and
 \texttt{psb-g}. In the left-hand plots in Figure~\ref{Fig-results}, we
 compare the performance of \texttt{psf-g}, \texttt{psf-r}, \texttt{psf-c}, and
 \texttt{psf-1}. This column of the figure demonstrates the superiority of the
 greedy variant (\texttt{psf-g}) and the usefulness of the block-iterative
 capabilities of projective splitting:  in particular, processing only one of
 the first $P$ blocks at each iteration, when this block is selected by the
 greedy heuristic as in \texttt{psf-g}, results in much better performance
 than the \texttt{psf-1} strategy of procesing the entire loss function at
 each iteration. \col{Further, the greedy heuristic outperforms both random 
 and cyclic selection.}



The right-hand plots in the figure compare \cp, \tseng, and \frb\space to our methods
\texttt{psf-g} and \texttt{psb-g}. These plots suggest that \tseng,
\frb, and \cp\space are not particularly competitive on this problem. 
Our method \psfg\space is the fastest method on
all examples. \col{Our similar method using approximate backward steps, \psbg, is very close in performance to \psfg\space for $\lambda=10^{-6}$, but is slower for $\lambda=10^{-8}$ and $\lambda=10^{-4}$.} Furthermore,
\psfg\space is arguably far simpler to implement than \psbg: for \psbg, one
must select a method for approximately solving the nonlinear program
\eqref{eqProxLR} at each iteration. While we chose L-BFGS, there are many other possibilities, each with its own parameters. For L-BFGS, we had
to choose the history parameter, the type of linesearch condition to use, and
other parameters. After making these choices, one then must implement
the subproblem solver; one might also be able to use some existing
implementation, but (in theory, at least) care must be taken to make sure that
it terminates using the proper stopping criteria~\eqref{err1}
and~\eqref{err2}. By contrast, the implementation details of
\psfg\space are contained within this manuscript
and fewer choices need to be made.
Overall, our experiments thus suggest that our new forward-step procedure can
improve the performance and usability of projective splitting.

\begin{figure}[h!]
	\includegraphics[width=0.49\linewidth]{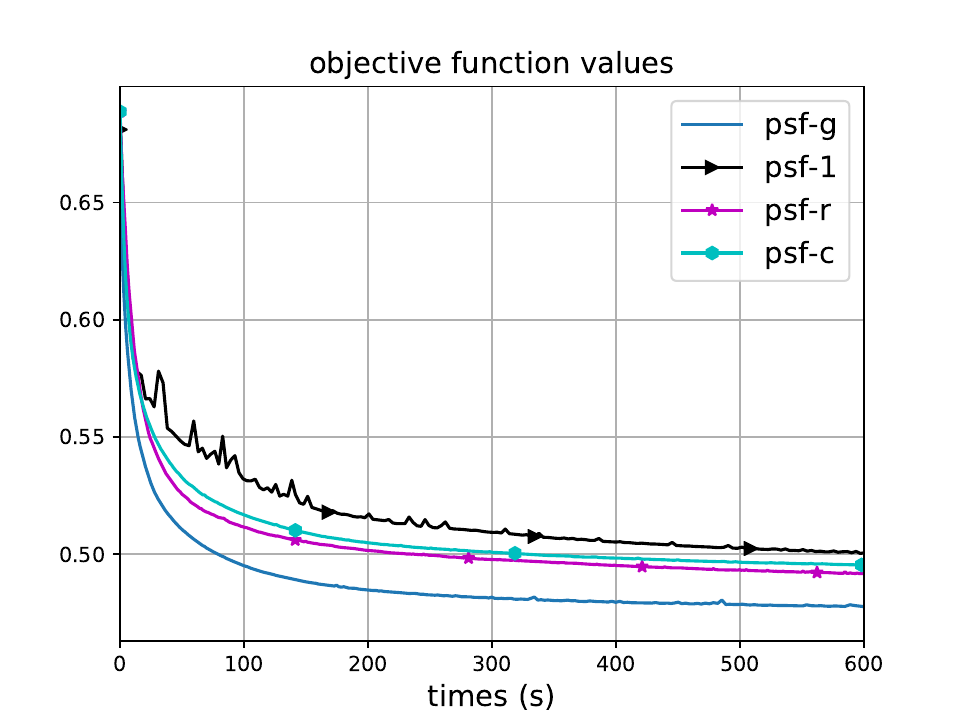}	
	\includegraphics[width=0.49\linewidth]{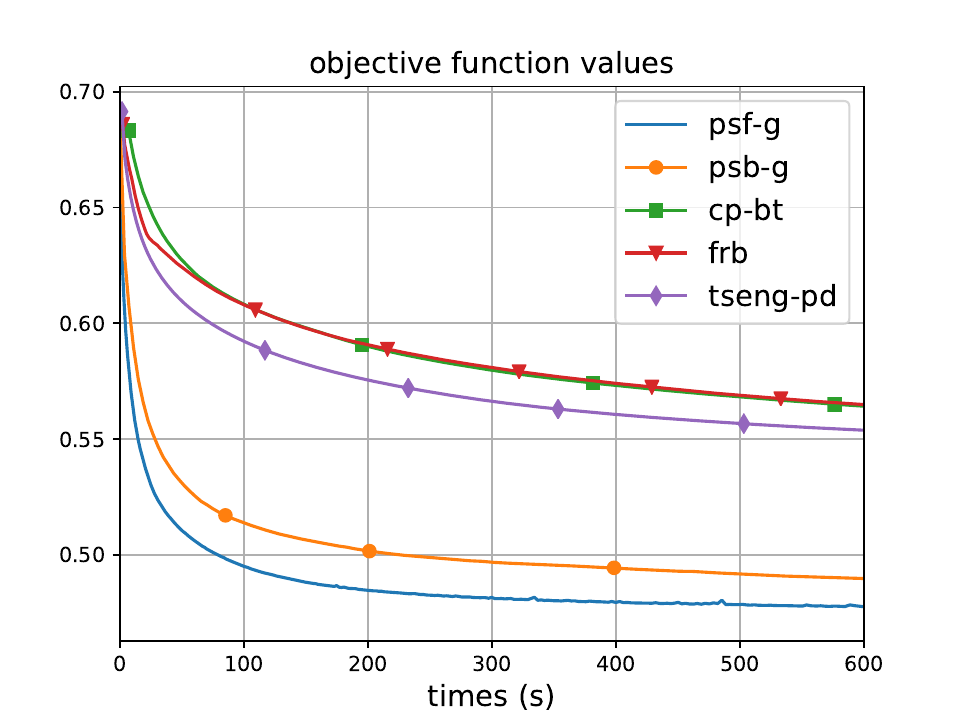}
	\includegraphics[width=0.49\linewidth]{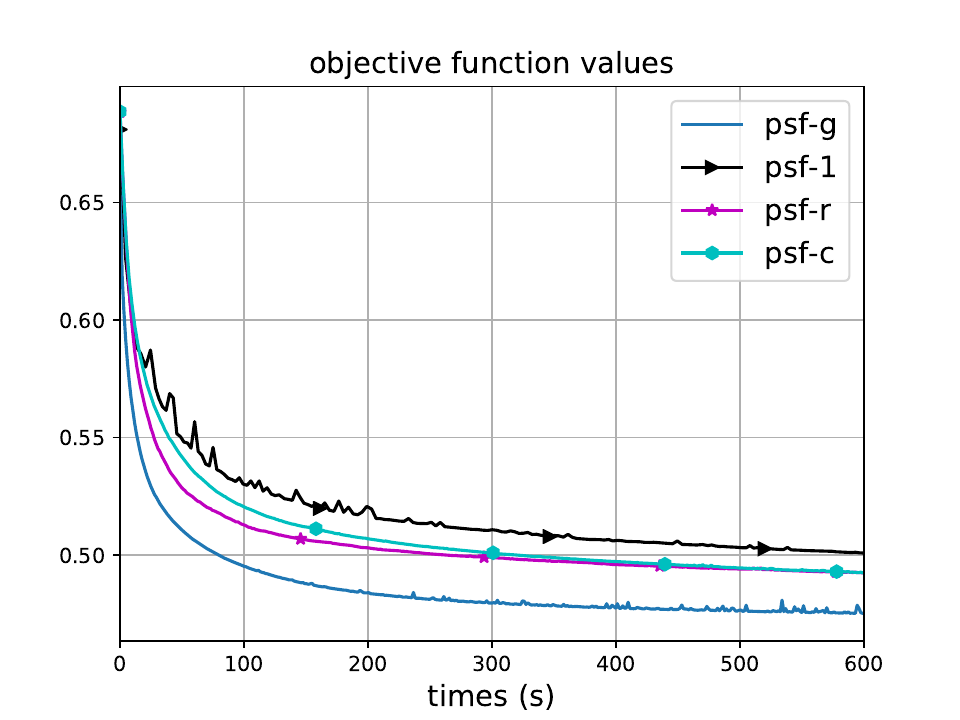}	
	\includegraphics[width=0.49\linewidth]{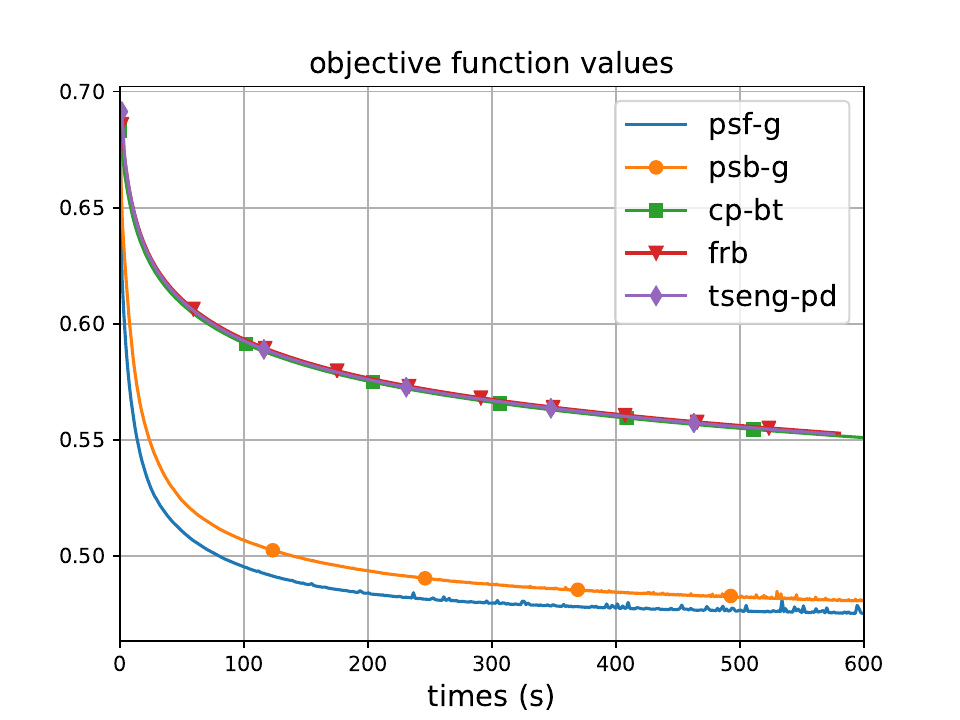}	
	\includegraphics[width=0.49\linewidth]{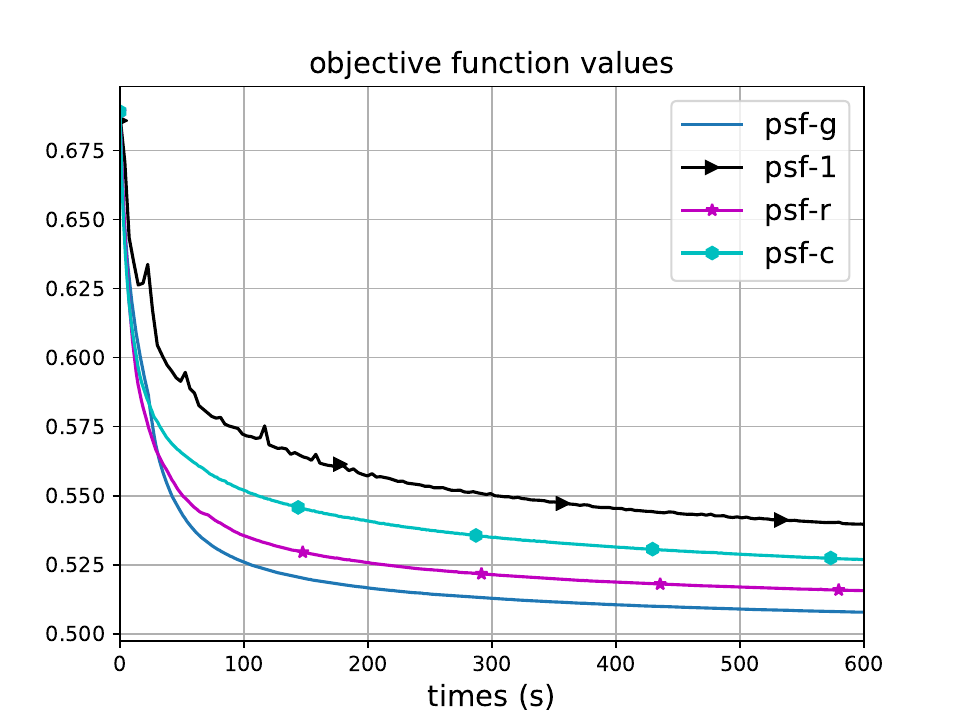}	
	\includegraphics[width=0.49\linewidth]{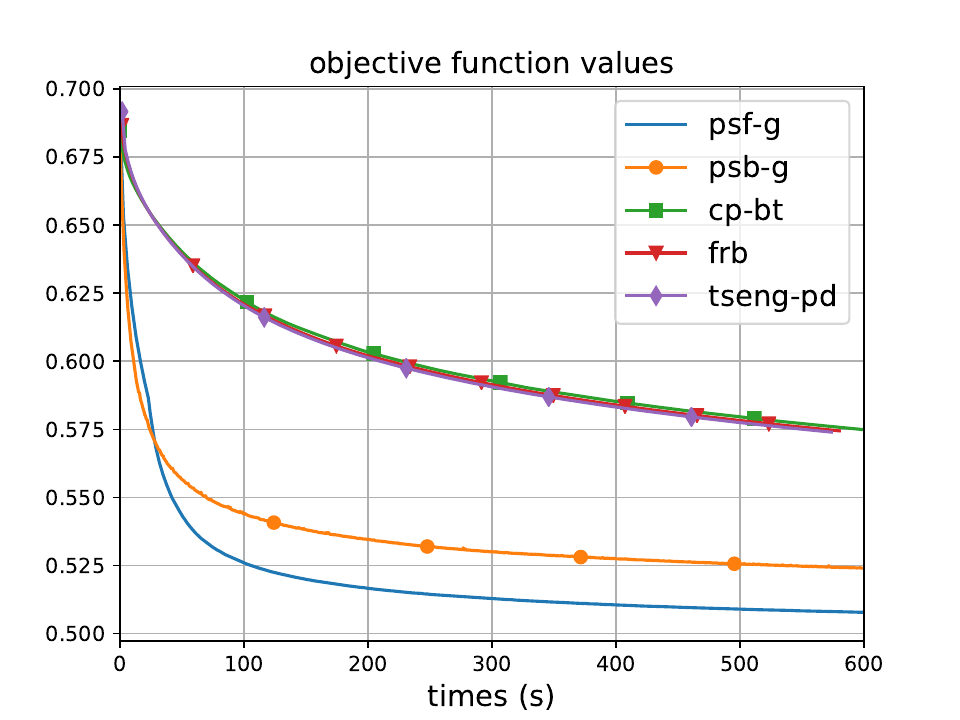}	
	\caption{\col{Objective values against wall-clock running time. Top row: $\lambda = 10^{-8}$, middle row: $\lambda=10^{-6}$, bottom row: $\lambda=10^{-4}$. }}	
	\label{Fig-results}
\end{figure}

\begin{acknowledgements}
This material is based upon work supported by the National Science 
Foundation under Grant No. 1617617. We thank Xiaohan Yan and Jacob Bien for kindly sharing their data for the
TripAdvisor reviews problem in Section \ref{secNumerical}.
\end{acknowledgements}
\bibliographystyle{spmpsci}
\bibliography{refs}

\end{document}